\documentclass[11pt,fleqn]{article}
\usepackage{geometry}                
\usepackage{graphicx}
\usepackage{amsthm,amsmath}
\usepackage{amssymb,dsfont,mathrsfs}

\numberwithin{equation}{section}
\newtheorem{theorem}{Theorem}[section]
\newtheorem{lemma}{Lemma}[section]
\newtheorem{proposition}{Proposition}[section]

\newtheorem{remark}{Remark}[section]
\newtheorem{assumption}{Assumption}[section]

\def\ba{\boldsymbol{a}}
\def\bb{\boldsymbol{b}}
\def\bc{\boldsymbol{c}}

\def\bg{\boldsymbol{g}}

\def\bi{\boldsymbol{i}}

\def\bl{\boldsymbol{l}}

\def\bt{\boldsymbol{t}}

\def\bx{\boldsymbol{x}}

\def\bnu{\boldsymbol{\nu}}
\def\btheta{\boldsymbol{\theta}}

\def\bzero{\mathbf{0}}
\def\bone{\mathbf{1}}

\def\calD{\mathcal{D}}

\def\scrI{\mathscr{I}}

\def\spr{\mbox{\rm spr}}
\def\cp{\mbox{\rm cp}}

\title{Asymptotics of the occupation measure defined on a nonnegative matrix of two-dimensional quasi-birth-and-death type}
\author{Toshihisa Ozawa  \\ 
Faculty of Business Administration, Komazawa University \\
1-23-1 Komazawa, Setagaya-ku, Tokyo 154-8525, Japan \\
E-mail: toshi@komazawa-u.ac.jp
}
\date{}

\begin{document}

\maketitle

\begin{abstract}
We consider a nonnegative matrix having the same block structure as that of the transition probability matrix of a  two-dimensional quasi-birth-and-death process (2d-QBD process for short) and define two kinds of measure for the nonnegative matrix. One corresponds to the mean number of visits to each state before the 2d-QBD process starting from the level zero returns to the level zero for the first time. The other corresponds to the probabilities that the 2d-QBD process starting from each state visits the level zero. We call the former the occupation measure and the latter the hitting measure. 
We obtain asymptotic properties of the occupation measure such as the asymptotic decay rate in an arbitrary direction. Those of the hitting measure can be obtained from the results for the occupation measure by using a kind of duality between the two measures. 
%


\smallskip
\textit{Keywards}: quasi-birth-and-death process, Markov additive process, occupation measure, asymptotic decay rate, matrix analytic method

\end{abstract}

%
%
\section{Introduction} \label{sec:intro}

We consider a nonnegative matrix having the same block structure as that of the transition probability matrix of a  two-dimensional quasi-birth-and-death process (2d-QBD process for short) and define two kinds of measure for the nonnegative matrix. One corresponds to the mean number of visits to each state before the 2d-QBD process starting from the level zero returns to the level zero for the first time. The other corresponds to the probabilities that the 2d-QBD process starting from each state visits the level zero. We call the former the occupation measure and the latter the hitting measure. 
The stationary distribution of a 2d-QBD process corresponds to the normalized occupation measure of the transition probability matrix, and its asymptotic properties have been studied in Ozawa \cite{Ozawa13}, Miyazawa \cite{Miyazawa15}, Ozawa and Kobayashi \cite{Ozawa18} and Ozawa \cite{Ozawa22,Ozawa23}, where the asymptotic decay rates and exact asymptotic functions of the stationary tail probabilities were obtained. In this paper, we demonstrate that the same results hold for the occupation measure defined on a nonnegative matrix of 2d-QBD type. Asymptotic properties of the hitting measure can be obtained from the results for the occupation measure by using a kind of duality between the two measures. 

In the following, we briefly explain our model and results. 
Denote by $\scrI_S$ the set of all the subsets of $\{1,2\}$, i.e., $\scrI_S=\{\emptyset,\{1\},\{2\},\{1,2\}\}$, and we use it as an index set. Divide $\mathbb{Z}_+^2$ into $2^2=4$ exclusive subsets defined as 
\[
\mathbb{B}^\alpha=\{\bx=(x_1,x_2)\in\mathbb{Z}_+^2; \mbox{$x_i>0$ for $i\in\alpha$, $x_i=0$ for $i\in\{1,2\}\setminus \alpha$} \},\ \alpha\in\scrI_S.  
\]
The class $\{\mathbb{B}^\alpha; \alpha\in\scrI_S\}$ is a partition of $\mathbb{Z}_+^2$. $\mathbb{B}^\emptyset$ is the set containing only the origin, and $\mathbb{B}^{\{1,2\}}$ is the set of all positive points in $\mathbb{Z}_+^2$. Let $S_0=\{1,2,...,s_0\}$ be a finite set. We call $S_0$ the set of \textit{phases}. 
%
\begin{figure}[t]
\begin{center}
\includegraphics[width=55mm,trim=0 0 0 0]{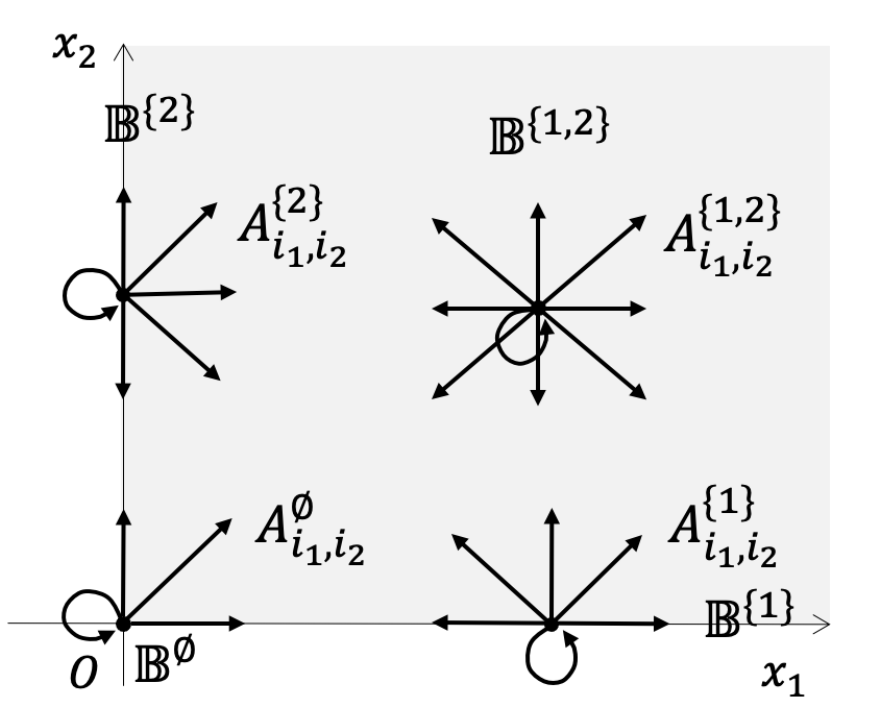} 
\caption{Transition probability blocks}
\label{fig:fig11}
\end{center}
\end{figure}
%
Let $T$ be a nonnegative matrix represented in block form as $T=\left( T_{\bx,\bx'}; \bx,\bx'\in\mathbb{Z}_+^2 \right)$, where $T_{\bx,\bx'}=(t_{(\bx,j),(\bx',j')}; j,j'\in S_0)$. The dimension of $T$ is $(\mathbb{Z}_+^2\times S_0)\times(\mathbb{Z}_+^2\times S_0)$. For fundamental definitions and properties with respect to nonnegative matrices, see Seneta \cite{Seneta06}.
We assume $T$ has the same block structure as that the transition probability matrix of a 2d-QBD process has. To be precise, for $\alpha\in\scrI_S$ and $i_1,i_2\in\{-1,0,1\}$, let $A^\alpha_{i_1,i_2}$ be an $s_0\times s_0$ nonnegative matrix, where $A^\emptyset_{i_1,-1}=A^\emptyset_{-1,i_2}=O$, $A^{\{1\}}_{i_1,-1}=O$ and $A^{\{2\}}_{-1,i_2}=O$ for every $i_1,i_2\in\{-1,0,1\}$, then $T_{\bx,\bx'}$ is given as 
\begin{equation}
T_{\bx,\bx'} 
= \left\{ \begin{array}{ll} 
A^\alpha_{\bx'-\bx}, & \mbox{if $\bx\in\mathbb{B}^\alpha$ for some $\alpha\in\scrI_S$ and $\bx'-\bx\in\{-1,0,1\}^2$}, \cr
O, & \mbox{otherwise},
\end{array} \right.
\end{equation}
where $O$ is a matrix of $0$'s whose dimension is determined in context. We assume the following condition throughout the paper. 
\begin{assumption} \label{as:QBD_irreducible}
The nonnegative matrix $T$ is irreducible and aperiodic. 
\end{assumption}

\begin{remark}
Since the number of positive elements of each row and column of $T$ is finite, we can apply the results for nonnegative matrices obtained in Ozawa \cite{Ozawa21} to $T$.
\end{remark}

%
Next, we define three nonnegative matrices induced from $T$: $T^{\{1\}}=(T^{\{1\}}_{\bx,\bx'};\bx,\bx'\in\mathbb{Z}\times\mathbb{Z}_+)$, $T^{\{2\}}=(T^{\{2\}}_{\bx,\bx'};\bx,\bx'\in\mathbb{Z}_+\times\mathbb{Z})$ and $T^{\{1,2\}}=(T^{\{1,2\}}_{\bx,\bx'};\bx,\bx'\in\mathbb{Z}\times\mathbb{Z})$, as follows.
\begin{align*}
&T^{\{1\}}_{\bx,\bx'} 
= \left\{ \begin{array}{ll} 
A^{\{1\}}_{\bx'-\bx}, & \mbox{if $\bx\in\mathbb{Z}\times\{0\}$ and $\bx'-\bx\in\{-1,0,1\}\times\{0,1\}$}, \cr
A^{\{1,2\}}_{\bx'-\bx}, & \mbox{if $\bx\in\mathbb{Z}\times\mathbb{N}$ and $\bx'-\bx\in\{-1,0,1\}^2$}, \cr
O, & \mbox{otherwise},
\end{array} \right. \\
&T^{\{2\}}_{\bx,\bx'} 
= \left\{ \begin{array}{ll} 
A^{\{2\}}_{\bx'-\bx}, & \mbox{if $\bx\in\{0\}\times\mathbb{Z}$ and $\bx'-\bx\in\{0,1\}\times\{-1,0,1\}$}, \cr
A^{\{1,2\}}_{\bx'-\bx}, & \mbox{if $\bx\in\mathbb{N}\times\mathbb{Z}$ and $\bx'-\bx\in\{-1,0,1\}^2$}, \cr
O, & \mbox{otherwise},
\end{array} \right. \\
&T^{\{1,2\}}_{\bx,\bx'} 
= \left\{ \begin{array}{ll} 
A^{\{1,2\}}_{\bx'-\bx}, & \mbox{if $\bx'-\bx\in\{-1,0,1\}^2$}, \cr
O, & \mbox{otherwise},
\end{array} \right.
\end{align*}
where $\mathbb{N}$ is the set of all positive integers. The nonnegative matrices $T^{\{1\}}$, $T^{\{2\}}$ and $T^{\{1,2\}}$ correspond to the transition probability matrices of the Markov additive processes induced from a 2d-QBD process, {\it see Section 1 of Ozawa \cite{Ozawa22}}. 
We assume the following condition throughout the paper. 
\begin{assumption} \label{as:MAprocess_irreducible}
The induced nonnegative matrices $T^{\{1\}}$, $T^{\{2\}}$ and $T^{\{1,2\}}$ are irreducible and aperiodic. 
\end{assumption}

We consider representations of Markov-additive-process-type (MAP-type for short) for $T^{\{1\}}$ and $T^{\{2\}}$. They are given by $\bar T^{\{1\}}=\left(\bar T^{\{1\}}_{x_1,x_1'};x_1,x_1'\in\mathbb{Z} \right)$ and $\bar T^{\{2\}}=\left(\bar T^{\{2\}}_{x_2,x_2'};x_2,x_2'\in\mathbb{Z} \right)$, where 
\begin{align*}
&\bar T^{\{1\}}_{x_1,x_1'} 
= \left\{ \begin{array}{ll} 
\bar A^{\{1\}}_{x_1'-x_1}, & \mbox{if $x_1'-x_1\in\{-1,0,1\}$}, \cr
O, & \mbox{otherwise},
\end{array} \right. \\[5pt]
&\bar A^{\{1\}}_i = \left( \bar A^{\{1\}}_{i,(x_2,x_2')}; x_2,x_2'\in\mathbb{Z}_+ \right),\\
&\bar A^{\{1\}}_{i,(x_2,x_2')} = \left\{ \begin{array}{ll}
 A^{\{1\}}_{i,x_2'-x_2}, & \mbox{if $x_2=0$ and $x_2'-x_2\in\{0,1\}$}, \cr
 A^{\{1,2\}}_{i,x_2'-x_2}, & \mbox{if $x_2\ge 1$ and $x_2'-x_2\in\{-1,0,1\}$}, \cr
 O, & \mbox{otherwise}, 
 \end{array} \right. 
 \end{align*}
 \begin{align*}
 &\bar T^{\{2\}}_{x_2,x_2'} 
= \left\{ \begin{array}{ll} 
\bar A^{\{2\}}_{x_2'-x_2}, & \mbox{if $x_2'-x_2\in\{-1,0,1\}$}, \cr
O, & \mbox{otherwise},
\end{array} \right. \\[5pt]
&\bar A^{\{2\}}_i = \left( \bar A^{\{2\}}_{i,(x_1,x_1')}; x_1,x_1'\in\mathbb{Z}_+ \right),\\
&\bar A^{\{2\}}_{i,(x_1,x_1')} = \left\{ \begin{array}{ll}
 A^{\{2\}}_{x_1'-x_1,i}, & \mbox{if $x_1=0$ and $x_1'-x_1\in\{0,1\}$}, \cr
 A^{\{1,2\}}_{x_1'-x_1,i}, & \mbox{if $x_1\ge 1$ and $x_1'-x_1\in\{-1,0,1\}$}, \cr
 O, & \mbox{otherwise}.
 \end{array} \right.
\end{align*}
Let $\bar A^{\{1\}}_*(z_1)$, $\bar A^{\{2\}}_*(z_2)$ and $A^{\{1,2\}}_{*,*}(z_1,z_2)$ be the matrix-valued generating functions defined as
\begin{align*}
&\bar A^{\{1\}}_*(z_1) = \sum_{i_1\in\{-1,0,1\}} z_1^{i_1} \bar A^{\{1\}}_{i_1},\quad 
\bar A^{\{2\}}_*(z_2) = \sum_{i_2\in\{-1,0,1\}} z_2^{i_2} \bar A^{\{2\}}_{i_2}, \\
&A^{\{1,2\}}_{*,*}(z_1,z_2) = \sum_{i_1,i_2\in\{-1,0,1\}} z_1^{i_1} z_2^{i_2} A^{\{1,2\}}_{i_1,i_2}. 
\end{align*}
Let $\Gamma^{\{1\}}$, $\Gamma^{\{2\}}$ and $\Gamma^{\{1,2\}}$ be regions in which  the convergence parameters of $\bar A^{\{1\}}_*(e^{\theta_1})$, $\bar A^{\{2\}}_*(e^{\theta_2})$ and $A^{\{1,2\}}_{*,*}(e^{\theta_1},e^{\theta_2})$ are greater than $1$, respectively, i.e., 
\begin{align*}
&\Gamma^{\{1\}} = \{(\theta_1,\theta_2)\in\mathbb{R}^2; \cp(\bar A^{\{1\}}_*(e^{\theta_1}))>1 \},\quad 
\Gamma^{\{2\}} = \{(\theta_1,\theta_2)\in\mathbb{R}^2; \cp(\bar A^{\{2\}}_*(e^{\theta_2}))>1 \}, \\
&\Gamma^{\{1,2\}} = \{(\theta_1,\theta_2)\in\mathbb{R}^2; \cp(A^{\{1,2\}}_{*,*}(e^{\theta_1},e^{\theta_2}))>1 \}, 
\end{align*}
where the convergence parameter of a nonnegative matrix $A$ with a finite or countable dimension is denoted by $\cp(A)$, i.e., $\cp(A) = \sup\{r\in\mathbb{R}_+; \sum_{n=0}^\infty r^n A^n<\infty,\ \mbox{entry-wise} \}$. 
We also define $\Gamma_0^{\{1\}}$ and $\Gamma_0^{\{2\}}$ as 
\begin{align*}
&\Gamma_0^{\{1\}} = \{\theta_1\in\mathbb{R}; \cp(\bar A^{\{1\}}_*(e^{\theta_1}))>1 \},\quad 
\Gamma_0^{\{2\}} = \{\theta_2\in\mathbb{R}; \cp(\bar A^{\{2\}}_*(e^{\theta_2}))>1 \}.
\end{align*}
By Lemma A.1 of  Ozawa \cite{Ozawa21}, $\cp(\bar A^{\{1\}}_*(e^\theta))^{-1}$ and $\cp(\bar A^{\{2\}}_*(e^\theta))^{-1}$ are log-convex in $\theta$, and the closures of $\Gamma^{\{1\}}$ and $\Gamma^{\{2\}}$ are convex sets; $\cp(\bar A^{\{1,2\}}_*(e^{\theta_1},e^{\theta_2}))^{-1}$ is also log-convex in $(\theta_1,\theta_2)$, and the closure of $\Gamma^{\{1,2\}}$ is a convex set. Furthermore, by Proposition B.1 of Ozawa \cite{Ozawa21}, $\Gamma^{\{1,2\}}$ is bounded under Assumption \ref{as:MAprocess_irreducible}. 
%
%
%
%
We also define the reverse of the regions as follows: for $\alpha\in\{\{1\},\{2\},\{1,2\}\}$, 
\[
\Gamma^\alpha_{R} =  \{(\theta_1,\theta_2)\in\mathbb{R}^2; (-\theta_1,-\theta_2)\in\Gamma^\alpha \}.
\]

%
For $(x_1,x_2)\in\mathbb{Z}^2$, define a set of states, $\mathbb{L}_{(x_1,x_2)}$, as 
\[
\mathbb{L}_{(x_1,x_2)} =\{(x_1,x_2,j); j\in S_0\}, 
\]
and call it level $(x_1,x_2)$. We call $\mathbb{L}_{(0,0)}$ the level zero. Let $\hat T$ be the nonnegative block matrix obtained by removing the $\mathbb{L}_{(0,0)}$-row and $\mathbb{L}_{(0,0)}$-column from $T$, i.e., 
\[
\hat T = \left( T_{\bx,\bx'}; \bx,\bx'\in\mathbb{Z}_+^2\setminus\{(0,0)\} \right).
\]
Furthermore, define nonnegative block vectors $\hat\bt_{01}$ and $\hat\bt_{10}$ as 
\[
\hat\bt_{01} = \left( T_{(0,0),\bx'}, \bx'\in\mathbb{Z}_+^2\setminus\{(0,0)\} \right),\quad 
\hat\bt_{10} = \left( T_{\bx,(0,0)}, \bx\in\mathbb{Z}_+^2\setminus\{(0,0)\} \right).
\]
We define the occupation measure $\hat\bnu=\left( \hat\bnu_{\bx}, \bx\in\mathbb{Z}_+^2\setminus\{(0,0)\} \right)$ of $T$, which corresponds to the mean number of visits to each state before a 2d-QBD process starting from the level zero returns to the level zero for the first time, as 
\begin{equation}
\hat\bnu = \hat\bt_{01} \hat\Phi = \hat\bt_{01} \sum_{n=0}^\infty \hat T^n, 
\label{eq:hatbnu_definition}
\end{equation}
where $\hat\Phi$ is the potential matrix of $\hat T$. For $\bx\in\mathbb{Z}_+^2\setminus\{(0,0)\}$, $\hat\bnu_{\bx}$ is represented as 
\[
\hat\bnu_{\bx}=\begin{pmatrix} \hat\nu_{(0,0,j_0)(\bx,j)}; j_0,j\in S_0 \end{pmatrix}.
\]
We also define the hitting measure $\hat\bg=\left( \hat\bg_{\bx}, \bx\in\mathbb{Z}_+^2\setminus\{(0,0)\} \right)$ of $T$, which corresponds to the probabilities that the 2d-QBD process starting from each state visits the level zero, as 
\begin{equation}
\hat\bg = \hat\Phi \hat\bt_{10}  = \sum_{n=0}^\infty \hat T^n\, \hat\bt_{10}. 
\label{eq:hatg_definition}
\end{equation}
For $\bx\in\mathbb{Z}_+^2\setminus\{(0,0)\}$, $\hat\bg_{\bx}$ is represented as 
\[
\hat\bg_{\bx}=\begin{pmatrix} \hat g_{(\bx,j),(0,0,j_0)}; j, j_0\in S_0 \end{pmatrix}. 
\]
We assume the following condition throughout the paper. 
\begin{assumption} \label{as:hatPhi_finite}
The potential matrix $\hat\Phi$ is entry-wise finite.
\end{assumption}

\begin{remark}
We do not consider conditions that ensure Assumption \ref{as:hatPhi_finite} hold true. We just consider what kinds of asymptotic property of the occupation measure and hitting measure hold if the potential matrix $\hat\Phi$ is entry-wise finite. 
\end{remark}
For $s_1,s_2\in\mathbb{R}$, define regions $\Lambda^{\{1\}}(s_1)$ and $\Lambda^{\{2\}}(s_2)$ as 
\[
\Lambda^{\{1\}}(s_1) = \{(\theta_1,\theta_2)\in\mathbb{R}^2; \theta_1<s_1\},\quad 
\Lambda^{\{2\}}(s_2) = \{(\theta_1,\theta_2)\in\mathbb{R}^2; \theta_2<s_2\}, 
\]
and consider the following optimization problem:
\begin{align*}
\begin{array}{ll}
\mbox{maximize} & s_1+s_2, \cr
\mbox{subject to} & s_1 \le \sup\left\{\theta_1\in\mathbb{R}; (\theta_1,\theta_2)\in \Gamma^{\{1\}}\cap\Lambda^{\{2\}}(s_2)\cap\Gamma^{\{1,2\}}\ \mbox{for some $\theta_2\in\mathbb{R}$} \right\}, \cr
&  s_2 \le \sup\left\{\theta_2\in\mathbb{R}; (\theta_1,\theta_2)\in\Lambda^{\{1\}}(s_1)\cap\Gamma^{\{2\}}\cap\Gamma^{\{1,2\}}\ \mbox{for some $\theta_1\in\mathbb{R}$} \right\}.
\end{array}
\end{align*}
If $\Gamma^{\{1\}}\cap\Gamma^{\{2\}}\cap\Gamma^{\{1,2\}}\ne\emptyset$, this optimization problem has a unique optimal solution $(s_1^*,s_2^*)$. 
Replacing $\Gamma^{\{1\}}$, $\Gamma^{\{2\}}$ and $\Gamma^{\{1,2\}}$ with $\Gamma_R^{\{1\}}$, $\Gamma_R^{\{2\}}$ and $\Gamma_R^{\{1,2\}}$, respectively, we also define the optimal solution $(s_{R,1}^*,s_{R,2}^*)$, analogously. 
%
In this paper, we demonstrate that, for every discrete direction vector $\bc\in\mathbb{Z}_+^2\setminus\{(0,0)\}$, the asymptotic decay rates of the occupation measure and hitting measure are given by 
\begin{align}
\lim_{n\to\infty} \frac 1 n \log \hat\nu_{(0,0,j_0),(n \bc,j)} 
&= -\min\Bigl\{ \sup\{\langle \bc,\btheta \rangle; \btheta\in\Lambda^{\{1\}}(s_1^*)\cap\Gamma^{\{1,2\}}\},\cr
&\qquad\qquad\qquad\qquad \sup\{\langle \bc,\btheta \rangle; \btheta\in\Lambda^{\{2\}}(s_2^*)\cap\Gamma^{\{1,2\}}\} \Bigr\}, \\
\lim_{n\to\infty} \frac 1 n \log \hat g_{(n \bc,j),(0,0,j_0)} 
&= -\min\Bigl\{ \sup\{\langle \bc,\btheta \rangle; \btheta\in\Lambda^{\{1\}}(s_{R,1}^*)\cap\Gamma^{\{1,2\}}_R\},\cr
&\qquad\qquad\qquad\qquad \sup\{\langle \bc,\btheta \rangle; \btheta\in\Lambda^{\{2\}}(s_{R,2}^*)\cap\Gamma^{\{1,2\}}_R\} \Bigr\}, 
\end{align}
where $\langle \ba,\bb \rangle$ is the inner product of vectors $\ba$ and $\bb$.
%
%
We prove them by using the same approach as that used in Ozawa \cite{Ozawa22,Ozawa23}. For related works, see Section 1 of \cite{Ozawa22}. 
Our model includes transient 2d-QBD processes. With respect to transient stochastic models, see, for example,  Franceschi et al.\ \cite{Franceschi24} and references therein. It deals with a transient reflected Brownian motion.

%
The rest of the paper is organized as follows. 
In Section \ref{sec:preliminary}, we introduce the potential matrices of the induced nonnegative matrices and give key equations that the occupation measure satisfy. We call them the compensation equations. 
In Section \ref{sec:convergence_domain}, we define several kinds of matrix-valued generating function for the occupation measure and give relations among the convergence domains of the matrix-valued generating functions. 
In Section \ref{sec:asymptotics_o}, we obtain the convergence domains of the matrix-valued generating functions for the occupation measure, see Theorems \ref{pr:calD1D2_eq} and \ref{pr:calD_eq}, and the asymptotic decay rates and functions for the occupation measure, see Theorems \ref{th:xi10xi01_eq} through \ref{th:h1c_eq}. 
In Section \ref{sec:asymptotics_h}, we consider a kind of duality between the occupation measure and hitting measure. 
The paper concludes with a remark about three-dimensional models in Section \ref{sec:conclusion}.

\medskip
\textit{Notation for vectors, matrices and sets.} 
%
%
For a matrix $A$, we denote by $[A]_{i,j}$ the $(i,j)$-entry of $A$ and by $A^\top$ the transpose of $A$.  
Similar notations are also used for vectors. 
%
%
For a finite square matrix $A$, we denote by $\spr(A)$ the spectral radius of $A$, which is the maximum modulus of eigenvalue of $A$. If $A$ is nonnegative, $\spr(A)$ corresponds to the Perron-Frobenius eigenvalue of $A$ and we have $\spr(A)=\cp(A)^{-1}$. 
%
$\bone$ is a column vector of $1$'s and $\bzero$ is a column vector of $0$'s; their dimensions, which are finite or countably infinite, are determined in context. $I$ is the identity matrix. 
%
%
%
For $n\in\mathbb{Z}$, $\mathbb{Z}_{\ge n}$ is the set of all integers grater than or equal to $n$, i.e., $\mathbb{Z}_{\ge n}=\{n,n+1,n+2, \cdots\}$. 
For a convex set $S$ on $\mathbb{R}^2$, $[S]^{ex}$ is an extension of $S$ defined as
\[
[S]^{ex} = \{ \bx\in\mathbb{R}^2; \mbox{$\bx<\bx'$ for some $\bx'\in S$} \}, 
\]
and $[S]_{\{1\}}$ and $[S]_{\{2\}}$ are the projections of $S$ defined as
\begin{align*}
&[S]_{\{1\}} = \{x_1\in\mathbb{R}; (x_1,x_2)\in S\ \mbox{for some $x_2\in\mathbb{R}$} \}, \\
&[S]_{\{2\}} = \{x_2\in\mathbb{R}; (x_1,x_2)\in S\ \mbox{for some $x_1\in\mathbb{R}$} \}.
\end{align*}
$\mathbb{C}$ is the set of all complex numbers. 
For $r\in\mathbb{R}_+$, $\Delta_r$ is the open disk of radius $r$ and center $0$ on $\mathbb{C}$. 
For $r_1,r_2\in\mathbb{R}_+$ such that $r_1<r_2$, $\Delta_{r_1,r_2}$ is the open annual domain on $\mathbb{C}$ defined as
\[
\Delta_{r_1,r_2} = \{z\in\mathbb{C}; r_1<|z|<r_2 \}.
\]
$\bar\Delta_{r_1,r_2}$ is the closure of $\Delta_{r_1,r_2}$.

%
%
\section{Preliminaries} \label{sec:preliminary}

%
\subsection{Potential matrices}

Let $\Phi^{\{1\}}=\begin{pmatrix}\Phi^{\{1\}}_{\bx,\bx'}\end{pmatrix}$, $\Phi^{\{2\}}=\begin{pmatrix}\Phi^{\{2\}}_{\bx,\bx'}\end{pmatrix}$ and $\Phi^{\{1,2\}}=\begin{pmatrix}\Phi^{\{1,2\}}_{\bx,\bx'}\end{pmatrix}$ be the potential matrices of $T^{\{1\}}$, $T^{\{2\}}$ and $T^{\{1,2\}}$, respectively, defined as
\[
\Phi^{\{1\}} = \sum_{n=0}^\infty (T^{\{1\}})^n,\quad 
\Phi^{\{2\}} = \sum_{n=0}^\infty (T^{\{2\}})^n,\quad 
\Phi^{\{1,2\}} = \sum_{n=0}^\infty (T^{\{1,2\}})^n, 
\]
and $\bar\Phi^{\{1\}}$ and $\bar\Phi^{\{2\}}$ the representations of MAP-type for $\Phi^{\{1\}}$ and $\Phi^{\{2\}}$, respectively, defined as 
\begin{align*}
\bar\Phi^{\{1\}} = \begin{pmatrix} \bar\Phi^{\{1\}}_{x_1,x_1'}; x_1,x_1'\in\mathbb{Z} \end{pmatrix},\quad 
\bar\Phi^{\{1\}}_{x_1,x_1'} = \begin{pmatrix} \Phi^{\{1\}}_{(x_1,x_2),(x_1',x_2')}; x_2,x_2'\in\mathbb{Z}_+ \end{pmatrix}, \\
\bar\Phi^{\{2\}} = \begin{pmatrix} \bar\Phi^{\{2\}}_{x_2,x_2'}; x_2,x_2'\in\mathbb{Z} \end{pmatrix},\quad 
\bar\Phi^{\{2\}}_{x_2,x_2'} = \begin{pmatrix} \Phi^{\{2\}}_{(x_1,x_2),(x_1',x_2')}; x_1,x_1'\in\mathbb{Z}_+ \end{pmatrix}.
\end{align*}
For $x_1,x_2\in\mathbb{Z}$, let $\bar\Phi^{\{1\}}_{x_1,*}(z)$, $\bar\Phi^{\{2\}}_{x_2,*}(w)$ and $\Phi^{\{1,2\}}_{(x_1,x_2),(*,*)}(z,w)$ be the matrix-valued generating functions defined as
\begin{align*}
&\bar\Phi^{\{1\}}_{x_1,*}(z) = \sum_{x_1'=-\infty}^\infty z^{x_1'} \bar\Phi^{\{1\}}_{x_1,x_1'},\quad 
\bar\Phi^{\{2\}}_{x_2,*}(w) = \sum_{x_2'=-\infty}^\infty w^{x_2'} \bar\Phi^{\{2\}}_{x_2,x_2'}, \\
&\Phi^{\{1,2\}}_{(x_1,x_2),(*,*)}(z,w) =\sum_{x_1'=-\infty}^\infty \sum_{x_2'=-\infty}^\infty z^{x_1'} w^{x_2'} \Phi^{\{1,2\}}_{(x_1,x_2),(x_1',x_2')}.
\end{align*}
Because of space-homogeneity, we have for every  $x_1,x_2\in\mathbb{Z}$ that
\begin{align}
&\bar\Phi^{\{1\}}_{x_1,*}(z) = z^{x_1} \bar\Phi^{\{1\}}_{0,*}(z),\\
&\bar\Phi^{\{2\}}_{x_2,*}(z) = w^{x_2} \bar\Phi^{\{2\}}_{0,*}(w),\\
&\Phi^{\{1,2\}}_{(x_1,x_2),(*,*)}(z,w) = z^{x_1} w^{x_2} \Phi^{\{1,2\}}_{(0,0),(*,*)}(z,w). \label{eq:Phi12_eq1}
\end{align}
By the definition of the matrix-valued generating functions, their convergence domains are given by $\Gamma_0^{\{1\}}$, $\Gamma_0^{\{2\}}$ and $\Gamma^{\{1,2\}}$, respectively, i.e., 
\begin{align}
&\Gamma_0^{\{1\}} = \mbox{the interior of } \{\theta_1\in\mathbb{R}; \bar\Phi^{\{1\}}_{0,*}(e^{\theta_1})<\infty,\ \mbox{entry-wise} \}, \label{eq:Gamma1_eq1} \\ 
&\Gamma_0^{\{2\}} = \mbox{the interior of } \{\theta_2\in\mathbb{R}; \bar\Phi^{\{2\}}_{0,*}(e^{\theta_2})<\infty,\ \mbox{entry-wise} \},  \label{eq:Gamma2_eq1} \\
&\Gamma^{\{1,2\}} = \mbox{the interior of } \{(\theta_1,\theta_2)\in\mathbb{R}^2; \Phi^{\{1,2\}}_{(0,0),(*,*)}(e^{\theta_1},e^{\theta_2})<\infty,\ \mbox{entry-wise} \}. 
\end{align}

%
\subsection{Compensation equations}

Let $\hat T^{\{1\}}=\begin{pmatrix}\hat T^{\{1\}}_{\bx,\bx'}; \bx,\bx'\in\mathbb{Z}\times\mathbb{Z}_+\end{pmatrix}$, $\hat T^{\{2\}}=\begin{pmatrix}\hat T^{\{2\}}_{\bx,\bx'}; \bx,\bx'\in\mathbb{Z}_+\times\mathbb{Z}\end{pmatrix}$ and $\hat T^{\{1,2\}}=\begin{pmatrix}\hat T^{\{1,2\}}_{\bx,\bx'}; \bx,\bx'\in\mathbb{Z}^2\end{pmatrix}$ be the hybrids of the nonnegative matrix $\hat T$ and the induced nonnegative matrices $T^{\{1\}}$, $T^{\{2\}}$ and $T^{\{1,2\}}$, respectively, defined as 
\begin{align*}
\hat T^{\{1\}}_{\bx,\bx'} 
&= \left\{ \begin{array}{ll} 
\hat T_{\bx,\bx'}, & \mbox{if $\bx,\bx'\in\mathbb{Z}_+^2\setminus\{(0,0)\}$}, \cr
T^{\{1\}}_{\bx,\bx'}, & \mbox{if $\bx\in\mathbb{Z}\setminus\mathbb{Z}_+\times\mathbb{Z}_+$ and $\bx'\in\mathbb{Z}\times\mathbb{Z}_+$}, \cr
O, & \mbox{otherwise},
\end{array} \right. 
\end{align*}
\begin{align*}
\hat T^{\{2\}}_{\bx,\bx'} 
&= \left\{ \begin{array}{ll} 
\hat T_{\bx,\bx'}, & \mbox{if $\bx,\bx'\in\mathbb{Z}_+^2\setminus\{(0,0)\}$}, \cr
T^{\{2\}}_{\bx,\bx'}, & \mbox{if $\bx\in\mathbb{Z}_+\times\mathbb{Z}\setminus\mathbb{Z}_+$ and $\bx'\in\mathbb{Z}_+\times\mathbb{Z}$}, \cr
O, & \mbox{otherwise},
\end{array} \right. 
\end{align*}
\begin{align*}
\hat T^{\{1,2\}}_{\bx,\bx'} 
&= \left\{ \begin{array}{ll} 
\hat T_{\bx,\bx'}, & \mbox{if $\bx,\bx'\in\mathbb{Z}_+^2\setminus\{(0,0)\}$}, \cr
T^{\{1,2\}}_{\bx,\bx'}, & \mbox{if $\bx\in\mathbb{Z}^2\setminus\mathbb{Z}_+^2$ and $\bx'\in\mathbb{Z}^2$}, \cr
O, & \mbox{otherwise}.
\end{array} \right. 
\end{align*}
To be precise, 
\begin{align*}
\hat T^{\{1\}}_{\bx,\bx'} 
&= \left\{ \begin{array}{ll} 
A^{\{1\}}_{\bx'-\bx}, & \mbox{if $\bx=(1,0)$ and $\bx'-\bx\in(\{-1,0,1\}\times\{0,1\})\setminus\{(-1,0)\}$}, \cr
A^{\{1\}}_{\bx'-\bx}, & \mbox{if $\bx\in(\mathbb{Z}\setminus\{0,1\})\times\{0\}$ and $\bx'-\bx\in\{-1,0,1\}\times\{0,1\}$}, \cr
A^{\{2\}}_{\bx'-\bx}, & \mbox{if $\bx=(0,1)$ and $\bx'-\bx\in(\{0,1\}\times\{-1,0,1\})\setminus\{(0,-1)\}$}, \cr
A^{\{2\}}_{\bx'-\bx}, & \mbox{if $\bx\in\{0\}\times\mathbb{Z}_{\ge 2}$ and $\bx'-\bx\in\{0,1\}\times\{-1,0,1\}$}, \cr
A^{\{1,2\}}_{\bx'-\bx}, & \mbox{if $\bx=(1,1)$ and $\bx'-\bx\in\{-1,0,1\}^2\setminus\{(-1,-1)\}$}, \cr
A^{\{1,2\}}_{\bx'-\bx}, & \mbox{if $\bx\in((\mathbb{Z}\setminus\{0\})\times\mathbb{N})\setminus\{(1,1)\}$ and $\bx'-\bx\in\{-1,0,1\}^2$}, \cr
O, & \mbox{otherwise},
\end{array} \right. 
\end{align*}
\begin{align*}
\hat T^{\{2\}}_{\bx,\bx'} 
&= \left\{ \begin{array}{ll} 
A^{\{1\}}_{\bx'-\bx}, & \mbox{if $\bx=(1,0)$ and $\bx'-\bx\in(\{-1,0,1\}\times\{0,1\})\setminus\{(-1,0)\}$}, \cr
A^{\{1\}}_{\bx'-\bx}, & \mbox{if $\bx\in\mathbb{Z}_{\ge 2}\times\{0\}$ and $\bx'-\bx\in\{-1,0,1\}\times\{0,1\}$}, \cr
A^{\{2\}}_{\bx'-\bx}, & \mbox{if $\bx=(0,1)$ and $\bx'-\bx\in(\{0,1\}\times\{-1,0,1\})\setminus\{(0,-1)\}$}, \cr
A^{\{2\}}_{\bx'-\bx}, & \mbox{if $\bx\in\{0\}\times(\mathbb{Z}\setminus\{0,1\})$ and $\bx'-\bx\in\{0,1\}\times\{-1,0,1\}$}, \cr
A^{\{1,2\}}_{\bx'-\bx}, & \mbox{if $\bx=(1,1)$ and $\bx'-\bx\in\{-1,0,1\}^2\setminus\{(-1,-1)\}$}, \cr
A^{\{1,2\}}_{\bx'-\bx}, & \mbox{if $\bx\in(\mathbb{N}\times(\mathbb{Z}\setminus\{0\}))\setminus\{(1,1)\}$ and $\bx'-\bx\in\{-1,0,1\}^2$}, \cr
O, & \mbox{otherwise},
\end{array} \right. 
\end{align*}
\begin{align*}
\hat T^{\{1,2\}}_{(\bx,\bx')} 
&= \left\{ \begin{array}{ll} 
A^{\{1\}}_{\bx'-\bx}, & \mbox{if $\bx=(1,0)$ and $\bx'-\bx\in(\{-1,0,1\}\times\{0,1\})\setminus\{(-1,0)\}$}, \cr
A^{\{1\}}_{\bx'-\bx}, & \mbox{if $\bx\in\mathbb{Z}_{\ge 2}\times\{0\}$ and $\bx'-\bx\in\{-1,0,1\}\times\{0,1\}$}, \cr
A^{\{2\}}_{\bx'-\bx}, & \mbox{if $\bx=(0,1)$ and $\bx'-\bx\in(\{0,1\}\times\{-1,0,1\})\setminus\{(0,-1)\}$}, \cr
A^{\{2\}}_{\bx'-\bx}, & \mbox{if $\bx\in\{0\}\times\mathbb{Z}_{\ge 2}$ and $\bx'-\bx\in\{0,1\}\times\{-1,0,1\}$}, \cr
A^{\{1,2\}}_{\bx'-\bx}, & \mbox{if $\bx=(1,1)$ and $\bx'-\bx\in\{-1,0,1\}^2\setminus\{(-1,-1)\}$}, \cr
A^{\{1,2\}}_{\bx'-\bx}, & \mbox{if $\bx\in\mathbb{Z}^2\setminus(\{0\}\times\mathbb{Z}_+\cup\mathbb{Z}_+\times\{0\}\cup\{(1,1)\})$ and $\bx'-\bx\in\{-1,0,1\}^2$}, \cr
O, & \mbox{otherwise}.
\end{array} \right. 
\end{align*}
Let $\hat\bnu^{\{1\}}=\begin{pmatrix}\hat \bnu^{\{1\}}_{\bx}; \bx\in\mathbb{Z}\times\mathbb{Z}_+\end{pmatrix}$, $\hat\bnu^{\{2\}}=\begin{pmatrix}\hat \bnu^{\{2\}}_{\bx}; \bx\in\mathbb{Z}_+\times\mathbb{Z}\end{pmatrix}$ and $\hat\bnu^{\{1,2\}}=\begin{pmatrix}\hat \bnu^{\{1,2\}}_{\bx}; \bx\in\mathbb{Z}^2\end{pmatrix}$ be the nonnegative block vectors defined as, for $\alpha\in\{\{1\},\{2\},\{1,2\}\}$, 
\begin{align*}
\hat\bnu^\alpha_{\bx} 
&= \left\{ \begin{array}{ll} 
\hat\bnu_{\bx}, & \bx\in\mathbb{Z}_+^2\setminus\{(0,0)\}, \cr
O, & \mbox{otherwise}. 
\end{array} \right. 
\end{align*}
Furthermore, let $\hat\bt^{\{1\}}_{01}=\begin{pmatrix}\hat \bt^{\{1\}}_{01,\bx}; \bx\in\mathbb{Z}\times\mathbb{Z}_+\end{pmatrix}$, $\hat\bt^{\{2\}}_{01}=\begin{pmatrix}\hat \bt^{\{2\}}_{01,\bx}; \bx\in\mathbb{Z}_+\times\mathbb{Z}\end{pmatrix}$ and $\hat\bt^{\{1,2\}}_{01}=\begin{pmatrix}\hat \bt^{\{1,2\}}_{01,\bx}; \bx\in\mathbb{Z}^2\end{pmatrix}$ be the nonnegative block vectors defined as, for $\alpha\in\{\{1\},\{2\},\{1,2\}\}$, 
\begin{align*}
\hat\bt^\alpha_{01,\bx} 
&= \left\{ \begin{array}{ll} 
T_{(0,0),\bx}, & \bx\in\mathbb{Z}_+^2\setminus\{(0,0)\}, \cr
O, & \mbox{otherwise}. 
\end{array} \right. 
\end{align*}
Then, we obtain equations called compensation ones in Ozawa \cite{Ozawa22}, as follows.
\begin{lemma} \label{le:conpensation_eq}
For $\alpha\in\{\{1\},\{2\},\{1,2\}\}$, 
\begin{equation}
\hat\bnu^\alpha = \hat\bnu^\alpha (\hat T^\alpha-T^\alpha) \Phi^\alpha + \hat\bt^\alpha_{01} \Phi^\alpha. 
\label{eq:compensation_eq}
\end{equation}
\end{lemma}
\begin{proof}
By \eqref{eq:hatbnu_definition}, we have 
\begin{equation}
\hat\bnu = \hat\bnu \hat T + \hat\bt_{01}.
\end{equation}
Hence, we obtain by the Fubini's theorem that 
\begin{align}
\hat\bnu^\alpha 
&= \hat\bnu^\alpha (I-T^\alpha) \Phi^\alpha \cr
&= \hat\bnu^\alpha (\hat T^\alpha-T^\alpha) \Phi^\alpha + \hat\bt^\alpha_{01} \Phi^\alpha. 
\end{align}
\end{proof}

%
By the definition of $\hat T^{\{1\}}$, $\hat T^{\{2\}}$ and $\hat T^{\{1,2\}}$, if $\bx\notin\{0\}\times\mathbb{Z}_+\cup\{(1,0),(1,1)\}$, then $T^{\{1\}}_{\bx,\bx'}=\hat T_{\bx,\bx'}$, if $\bx\notin\mathbb{Z}_+\times\{0\}\cup\{(0,1),(1,1)\}$, then $T^{\{2\}}_{\bx,\bx'}=\hat T_{\bx,\bx'}$, and if $\bx\notin\{0\}\times\mathbb{Z}_+\cup\mathbb{Z}_+\times\{0\}\cup\{(1,1)\}$, then $T^{\{1,2\}}_{\bx,\bx'}=\hat T_{\bx,\bx'}$. Hence, we obtain from Lemma \ref{le:conpensation_eq} that, for $\bx\in\mathbb{Z}\times\mathbb{Z}_+$, 
\begin{align}
\hat\bnu^{\{1\}}_{\bx} 
&= -\hat\bnu_{(1,0)} A^{\{1\}}_{-1,0} \Phi^{\{1\}}_{(0,0),\bx} -\hat\bnu_{(1,1)} A^{\{1,2\}}_{-1,-1} \Phi^{\{1\}}_{(0,0),\bx} \cr
&\quad + \sum_{k=1}^\infty\ \sum_{i_1,i_2\in\{-1,0,1\}} \hat\bnu_{(0,k)} (A^{\{2\}}_{i_1,i_2}-A^{\{1,2\}}_{i_1,i_2}) \Phi^{\{1\}}_{(i_1,k+i_2),\bx} - \hat\bnu_{(0,1)} A^{\{2\}}_{0,-1} \Phi^{\{1\}}_{(0,0),\bx},
\label{eq:hatnux1} 
\end{align}
for $\bx\in\mathbb{Z}_+\times\mathbb{Z}$, 
\begin{align}
\hat\bnu^{\{2\}}_{\bx} 
&= -\hat\bnu_{(0,1)} A^{\{2\}}_{0,-1} \Phi^{\{2\}}_{(0,0),\bx} -\hat\bnu_{(1,1)} A^{\{1,2\}}_{-1,-1} \Phi^{\{2\}}_{(0,0),\bx} \cr
&\quad + \sum_{k=1}^\infty\ \sum_{i_1,i_2\in\{-1,0,1\}} \hat\bnu_{(k,0)} (A^{\{1\}}_{i_1,i_2}-A^{\{1,2\}}_{i_1,i_2}) \Phi^{\{2\}}_{(k+i_1,i_2),\bx} - \hat\bnu_{(1,0)} A^{\{1\}}_{-1,0} \Phi^{\{2\}}_{(0,0),\bx},
\label{eq:hatnux2}
\end{align}
and for $\bx\in\mathbb{Z}\times\mathbb{Z}$, 
\begin{align}
\hat\bnu^{\{1,2\}}_{\bx} 
& = -\hat\bnu_{(1,1)} A^{\{1,2\}}_{-1,-1} \Phi^{\{1,2\}}_{(0,0),\bx} \cr
&\quad + \sum_{k=1}^\infty\ \sum_{i_1,i_2\in\{-1,0,1\}} \hat\bnu_{(0,k)} (A^{\{2\}}_{i_1,i_2}-A^{\{1,2\}}_{i_1,i_2}) \Phi^{\{1,2\}}_{(i_1,k+i_2),\bx} - \hat\bnu_{(0,1)} A^{\{2\}}_{0,-1} \Phi^{\{1,2\}}_{(0,0),\bx} \cr
&\quad + \sum_{k=1}^\infty\ \sum_{i_1,i_2\in\{-1,0,1\}} \hat\bnu_{(k,0)} (A^{\{1\}}_{i_1,i_2}-A^{\{1,2\}}_{i_1,i_2}) \Phi^{\{1,2\}}_{(k+i_1,i_2),\bx} - \hat\bnu_{(1,0)} A^{\{1\}}_{-1,0} \Phi^{\{1,2\}}_{(0,0),\bx}.
\label{eq:hatnux12}
\end{align}
Recall that every $A^{\{1\}}_{i_1,i_2}$ and $A^{\{2\}}_{i_1,i_2}$ corresponding to impossible transitions are assumed to be zero. Equations \eqref{eq:hatnux1} through \eqref{eq:hatnux12} play a crucial role in the following section.

%
%
\subsection{Asymptotic decay rates} \label{sec:decay_rate}

Let $\bc=(c_1,c_2)\in\mathbb{Z}_+^2\setminus\{(0,0)\}$ be an arbitrary discrete direction vector. For $j_0\in S_0$ and for $(\bx,j)\in\mathbb{Z}_+^2\times S_0$, define lower  and upper asymptotic decay rates $\underline{\xi}_{\bc,j_0}(\bx,j)$ and $\bar\xi_{\bc,j_0}(\bx,j)$ as 
\[
\underline{\xi}_{\bc,j_0}(\bx,j) = -\limsup_{n\to\infty} \frac{1}{n} \log \hat\nu_{(\bzero,j_0),(\bx+n\bc,j)},\quad 
\bar\xi_{\bc,j_0}(\bx,j) = -\liminf_{n\to\infty} \frac{1}{n} \log \hat\nu_{(\bzero,j_0),(\bx+n\bc,j)}.
\]
By the Cauchy-Hadamard theorem, the radius of convergence of the power series of the sequence $\{\hat\nu_{(\bzero,j_0),(\bx+n\bc,j)}\}_{n\ge 0}$ is given by $e^{\underline{\xi}_{\bc,j_0}(\bx,j)}$. If $\underline{\xi}_{\bc,j_0}(\bx,j)=\bar\xi_{\bc,j_0}(\bx,j)$, we denote them by $\xi_{\bc,j_0}(\bx,j)$ and call it the asymptotic decay rate. 
Under Assumption \ref{as:MAprocess_irreducible}, the following property holds.
\begin{proposition} \label{pr:xi_equality}
For every $(\bx,j),(\bx',j')\in\mathbb{Z}_+^2\times S_0$, $\underline{\xi}_{\bc,j_0}(\bx,j)=\underline{\xi}_{\bc,j_0}(\bx',j')$ and $\bar\xi_{\bc,j_0}(\bx,j)=\bar\xi_{\bc,j_0}(\bx',j')$. 
\end{proposition}

For $\bc\in\mathbb{N}^2$, this proposition can be proved in a manner similar to that used in the proof of Proposition 2.3 of Ozawa \cite{Ozawa22}. In the case where  $\bc=(1,0)$ and $\bc=(0,1)$, using $\bar A^{\{1\}}_i$ and $\bar A^{\{2\}}_i$ instead of $A^{\{1,2\}}_{i_1,i_2}$, respectively, it can also be proved. We, therefore, omit the proof of Proposition \ref{pr:xi_equality}. 
Hereafter, we denote $\underline{\xi}_{\bc,j_0}(\bx,j)$, $\bar\xi_{\bc,j_0}(\bx,j)$ and $\xi_{\bc,j_0}(\bx,j)$ by $\underline{\xi}_{\bc,j_0}$, $\bar\xi_{\bc,j_0}$ and $\xi_{\bc,j_0}$, respectively.

%
Let $\check T^{\{1\}}$ be a nonnegative block matrix derived from $T^{\{1\}}$ by restricting the state space of the level to $\mathbb{N}\times\mathbb{Z}_+$, i.e., $\check T^{\{1\}}=\begin{pmatrix}T^{\{1\}}_{\bx,\bx'};\bx,\bx'\in\mathbb{N}\times\mathbb{Z}_+\end{pmatrix}$. Analogously define $\check T^{\{2\}}$ and $\check T^{\{1,2\}}$, i.e., $\check T^{\{2\}}=\begin{pmatrix}T^{\{2\}}_{\bx,\bx'};\bx,\bx'\in\mathbb{Z}_+\times\mathbb{N}\end{pmatrix}$ and $\check T^{\{1,2\}}=\begin{pmatrix}T^{\{1,2\}}_{\bx,\bx'};\bx,\bx'\in\mathbb{N}^2\end{pmatrix}$.
For $\alpha\in\{ \{1\},\{2\},\{1,2\} \}$, let $\check{\Phi}^\alpha=\begin{pmatrix}\check{\Phi}^\alpha_{\bx,\bx'}\end{pmatrix}$ be the potential matrix of $\check T^\alpha$, i.e., 
\[
\check{\Phi}^\alpha = \sum_{n=0}^\infty \left(\check T^\alpha\right)^n.
\]
We assume the following condition throughout the paper.
\begin{assumption} \label{as:Y12_onZpZp_irreducible}
$\check T^{\{1\}}$, $\check T^{\{2\}}$ and $\check T^{\{1,2\}}$ are irreducible and aperiodic. 
\end{assumption}
This condition implies Assumption \ref{as:MAprocess_irreducible}. 
By Theorem 5.1 of Ozawa \cite{Ozawa21}, we have, for any direction vector $\bc\in\mathbb{N}^2$, every $\bx=(x_1,x_2)\in\mathbb{N}^2$ such that $x_1=1$ or $x_2=1$, every $\bl\in\mathbb{Z}_+^2$ and  every $j_1,j_2\in S_0$, 
\begin{equation}
\lim_{n\to\infty} \frac{1}{n} \log\, [\check{\Phi}^{\{1,2\}}_{\bx,n\bc+\bl}]_{j_1,j_2} = - \sup\{\langle \bc,\btheta \rangle; \btheta\in\Gamma^{\{1,2\}}\}. 
\label{eq:tildePhi_limit}
\end{equation}
Analogous results also hold in the case where  $\bc=(1,0)$ and $\bc=(0,1)$. Hence, we immediately obtain the following.
\begin{lemma} \label{le:decayrate_upper}
We have, for every $j_0\in S_0$,  
\begin{align}  
&\underline{\xi}_{(1,0),j_0}\le\bar\xi_{(1,0),j_0} \le \sup\{\theta_1\in\mathbb{R};(\theta_1,\theta_2)\in\Gamma^{\{1\}}\ \mbox{for some $\theta_2\in\mathbb{R}$} \}, \label{eq:xic_upper1}\\
&\underline{\xi}_{(0,1),j_0}\le\bar\xi_{(0,1),j_0} \le \sup\{\theta_2\in\mathbb{R};(\theta_1,\theta_2)\in\Gamma^{\{2\}}\ \mbox{for some $\theta_1\in\mathbb{R}$} \}, \label{eq:xic_upper2}
\end{align}
and for $\bc\in\mathbb{N}^2$, 
\begin{equation}  
\underline{\xi}_{\bc,j_0}\le\bar\xi_{\bc,j_0} \le \sup\{\langle \bc,\btheta \rangle; \btheta\in\Gamma^{\{1,2\}}\}.  
\label{eq:xic_upper}
\end{equation}
\end{lemma}

\begin{remark}
For $\bc\in\mathbb{Z}_+\setminus\{(0,0)\}$ and $j_0\in S_0$, if $\underline{\xi}_{\bc,j_0}=\bar\xi_{\bc,j_0}$, we have for every $n\in\mathbb{N}$ that $\xi_{n\bc,j_0}=n \xi_{\bc,j_0}$. 
\end{remark}

\begin{remark}
 For $\bc\in\mathbb{Z}_+^2\setminus\{(0,0)\}$ and $j_0,j\in S_0$, let $f_{\bc,j_0,j}(z)$ be the generating function of the sequence $\{\hat\nu_{(\bzero,j_0),(n\bc,j)}\}_{n\ge 1}$, i.e., $f_{\bc,j_0,j}(z)=\sum_{n=1}^\infty z^n \hat\nu_{(\bzero,j_0),(n\bc,j)}$. In Section \ref{sec:asymptotics_o}, we will see that the point $z=e^{\underline{\xi}_{\bc,j_0}}$ is a pole or algebraic branch point of the complex function $f_{\bc,j_0,j}(z)$. This leads us to $\underline{\xi}_{\bc,j_0}=\bar\xi_{\bc,j_0}$. 
\end{remark}

For $\bc\in\mathbb{Z}_+\setminus\{(0,0)\}$, let $\underline{\xi}_{\bc}$ be the minimum of $\underline{\xi}_{\bc,j_0}$ in $j_0\in S_0$ and $\xi_{\bc}$ that of $\xi_{\bc,j_0}$ in $j_0\in S_0$ if it exists, i.e., 
\[
\underline{\xi}_{\bc} = \min_{j_0\in S_0} \underline{\xi}_{\bc,j_0},\quad 
\xi_{\bc} = \min_{j_0\in S_0} \xi_{\bc,j_0}. 
\]
In Section \ref{sec:asymptotics_o}, it will be shown that the value of $\underline{\xi}_{\bc,j_0}$ does not depend on $j_0$.

%
%

\section{Convergence domains} \label{sec:convergence_domain}

%
\subsection{Matrix-valued generating functions of the occupation measure}

Let $\hat\bnu_*(z,w)$ be the matrix-valued generating function of the occupation measure defined as 
\[
\hat\bnu_*(z,w) = \sum_{(x_1,x_2)\in\mathbb{Z}_+^2\setminus\{(0,0)\}} z^{x_1} w^{x_2} \hat\bnu_{(x_1,x_2)}, 
\]
which is decomposed as
\begin{equation}
\hat\bnu_*(z,w) = \hat\bnu_{(*,0)}(z) + \hat\bnu_{(0,*)}(w) + \hat\bnu_{(*,*)}(z,w), 
\label{eq:hatnus_eq}
\end{equation}
where
\begin{align*}
&\hat\bnu_{(*,0)}(z) = \sum_{x_1=1}^\infty z^{x_1} \hat\bnu_{(x_1,0)},\quad 
\hat\bnu_{(0,*)}(w) = \sum_{x_2=1}^\infty w^{x_2} \hat\bnu_{(0,x_2)}, \\
&\hat\bnu_{(*,*)}(z,w) = \sum_{x_1=1}^\infty \sum_{x_2=1}^\infty z^{x_1} w^{x_2} \hat\bnu_{(x_1,x_2)}.
\end{align*}
Let $\calD^{\{1\}}$, $\calD^{\{2\}}$, $\calD^{\{1,2\}}$ and $\calD$ be the convergence domains of the generating functions defined as 
\begin{align*}
&\calD^{\{1\}} = \mbox{the interior of}\ \{(\theta_1,\theta_2)\in\mathbb{R}^2; \hat\bnu_{(*,0)}(e^{\theta_1})<\infty,\ \mbox{entry-wise} \}, \\
&\calD^{\{2\}} = \mbox{the interior of}\ \{(\theta_1,\theta_2)\in\mathbb{R}^2; \hat\bnu_{(0,*)}(e^{\theta_2})<\infty,\ \mbox{entry-wise} \}, \\
&\calD^{\{1,2\}} = \mbox{the interior of}\ \{(\theta_1,\theta_2)\in\mathbb{R}^2; \hat\bnu_{(*,*)}(e^{\theta_1},e^{\theta_2})<\infty,\ \mbox{entry-wise} \}, \\
&\calD = \mbox{the interior of}\ \{(\theta_1,\theta_2)\in\mathbb{R}^2; \hat\bnu_*(e^{\theta_1},e^{\theta_2})<\infty,\ \mbox{entry-wise} \},
\end{align*}
where we use moment-generating-function-type expression so that the defined domains become convex sets. We also define $\calD_0^{\{1\}}$ and $\calD_0^{\{2\}}$ as
\begin{align*}
&\calD_0^{\{1\}} = \mbox{the interior of}\ \{\theta_1\in\mathbb{R}; \hat\bnu_{(*,0)}(e^{\theta_1})<\infty,\ \mbox{entry-wise} \}, \\
&\calD_0^{\{2\}} = \mbox{the interior of}\ \{\theta_2\in\mathbb{R}; \hat\bnu_{(0,*)}(e^{\theta_2})<\infty,\ \mbox{entry-wise} \}.
\end{align*}
By Proposition \ref{pr:xi_equality}, we have $\calD^{\{1,2\}}\subset\calD^{\{1\}}$ and $\calD^{\{1,2\}}\subset\calD^{\{2\}}$. Hence, by \eqref{eq:hatnus_eq}, we have 
\begin{equation}
\calD=\calD^{\{1\}}\cap\calD^{\{2\}}\cap\calD^{\{1,2\}}=\calD^{\{1,2\}}. 
\end{equation}
Furthermore, in a manner similar to that used in the proof of Theorem 5.2 of Ozawa \cite{Ozawa21}, we obtain by Lemma \ref{le:decayrate_upper} the following. 
\begin{proposition} \label{le:domain_upper}
\begin{align}
&\calD_0^{\{1\}} \subset [\Gamma_0^{\{1\}}]^{ex}, \\
&\calD_0^{\{2\}} \subset [\Gamma_0^{\{2\}}]^{ex}, \\
&\calD^{\{1,2\}} \subset [\Gamma^{\{1\}}]^{ex}\cap[\Gamma^{\{2\}}]^{ex}\cap[\Gamma^{\{1,2\}}]^{ex}. 
\end{align}
\end{proposition}

In the following subsections, we obtain relations among the domains.

%
\subsection{Matrix-valued generating functions of the potential matrices}  \label{sec:mgf_potential}

Let $\Phi^{\{1\}}_{\bx,(*,0)}(z)$ and $\Phi^{\{2\}}_{\bx,(0,*)}(w)$ be the matrix-valued generating functions of the potential matrices defined as 
\begin{align*}
&\Phi^{\{1\}}_{\bx,(*,0)}(z) = \sum_{x_1=-\infty}^\infty z^{x_1} \Phi^{\{1\}}_{\bx,(x_1,0)},\quad 
\Phi^{\{2\}}_{\bx,(0,*)}(w) = \sum_{x_2=-\infty}^\infty w^{x_2} \Phi^{\{2\}}_{\bx,(0,x_2)}, 
\end{align*}
where $\bx\in\mathbb{Z}\times\mathbb{Z}_+$ for $\Phi^{\{1\}}_{\bx,(*,0)}(z)$ and $\bx\in\mathbb{Z}_+\times\mathbb{Z}$ for $\Phi^{\{2\}}_{\bx,(0,*)}(w)$.  The matrix-valued generating function $\Phi^{\{1,2\}}_{\bx,(*,*)}(z,w)$ have already been defined in Section \ref{sec:preliminary}. 
Because of space-homogeneity, we have
\begin{equation}
\Phi^{\{1\}}_{(x_1,x_2),(*,0)}(z) = z^{x_1} \Phi^{\{1\}}_{(0,x_2),(*,0)}(z),\quad 
\Phi^{\{2\}}_{(x_1,x_2),(0,*)}(w) = w^{x_2} \Phi^{\{1\}}_{(x_1,0),(0,*)}(w). 
\label{eq:Phi1and2_eq1}
\end{equation}

%
According to Ozawa and Kobayashi \cite{Ozawa18}, we define matrices corresponding to ones called G-matrices in the queueing theory. 
For $i,j\in\{-1,0,1\}$, define the following matrix-valued generating functions:
\begin{align*}
&A^{\{1,2\}}_{*,j}(z) = \sum_{i'\in\{-1,0,1\}} z^{i'} A^{\{1,2\}}_{i',j},\quad 
A^{\{1,2\}}_{i,*}(w) = \sum_{j'\in\{-1,0,1\}} w^{j'} A^{\{1,2\}}_{i,j'}. 
\end{align*}
Recall that the matrix-valued generating function $A^{\{1,2\}}_{*,*}(z,w)$ has already been defined in Section \ref{sec:intro}. 
%
%
%
Let $\chi(z,w)$ be the spectral radius of $A^{\{1,2\}}_{*,*}(z,w)$, i.e., $\chi(z,w)=\spr(A^{\{1,2\}}_{*,*}(z,w))$. Since $\chi(z,w)$ is the reciprocal of the convergence parameter of $A^{\{1,2\}}_{*,*}(z,w)$ when $z$ and $w$ are positive real numbers, we have  
\[
\Gamma^{\{1,2\}}=\{(\theta_1,\theta_2)\in\mathbb{R}^2; \chi(e^{\theta_1},e^{\theta_2})<1\}.
\]
Under Assumption \ref{as:MAprocess_irreducible}, $A^{\{1,2\}}_{*,*}(1,1)$ is irreducible and aperiodic. Furthermore, in a manner similar to that used in the proof of Lemma 2.2 of \cite{Ozawa18}, we see that $\Gamma^{\{1,2\}}$ is bounded. Since $\chi(e^{\theta_1},e^{\theta_2})$ is convex in $(\theta_1, \theta_2)\in\mathbb{R}^2$, the closure of $\Gamma^{\{1,2\}}$ is a convex set. 
For $i\in\{1,2\}$, let $\bar\theta^{\{1,2\}}_i$ and $\underline{\theta}^{\{1,2\}}_i$ be the extreme points of $\Gamma^{\{1,2\}}$ defined as 
\begin{align*}
&\bar\theta^{\{1,2\}}_i = \sup\{\theta_i\in\mathbb{R}; (\theta_1,\theta_2)\in\Gamma^{\{1,2\}}\ \mbox{for some $\theta_{3-i}\in\mathbb{R}$} \}, \\
&\underline{\theta}^{\{1,2\}}_i = \inf\{\theta_i\in\mathbb{R}; (\theta_1,\theta_2)\in\Gamma^{\{1,2\}}\ \mbox{for some $\theta_{3-i}\in\mathbb{R}$} \}. 
\end{align*}
For $\theta_1\in[\underline{\theta}^{\{1,2\}}_1, \bar\theta^{\{1,2\}}_1]$, let $\underline{\eta}_2(\theta_1)$ and $\bar\eta_2(\theta_1)$ be the two real solutions to equation $\chi(e^{\theta_1},e^{\theta_2})=1$, counting multiplicity, where $\underline{\eta}_2(\theta_1)\le\bar\eta_2(\theta_1)$. We analogously define $\underline{\eta}_1(\theta_2)$ and $\bar\eta_1(\theta_2)$.
%
For $n\ge 1$, define the following set of index sequences: 
\begin{align*}
&\scrI_n = \biggl\{\bi_{(n)}\in\{-1,0,1\}^n;\ \sum_{l=1}^k i_l\ge 0\ \mbox{for $k\in\{1,2,...,n-1\}$}\ \mbox{and} \sum_{l=1}^n i_l=-1 \biggr\}, 
\end{align*}
where $\bi_{(n)}=(i_1,i_2,...,i_n)$, and define the following matrix functions: 
\begin{align*}
&D_{1,n}(z) = \sum_{\bi_{(n)}\in\scrI_n} A^{\{1,2\}}_{*,i_1}(z) A^{\{1,2\}}_{*,i_2}(z) \cdots A^{\{1,2\}}_{*,i_n}(z),\\
&D_{2,n}(w) = \sum_{\bi_{(n)}\in\scrI_n} A^{\{1,2\}}_{i_1,*}(w) A^{\{1,2\}}_{i_2,*}(w) \cdots A^{\{1,2\}}_{i_n,*}(w).
\end{align*}
Define matrix functions $G_1(z)$ and $G_2(w)$ as 
\begin{align*}
G_1(z) = \sum_{n=1}^\infty D_{1,n}(z),\quad 
G_2(w) = \sum_{n=1}^\infty D_{2,n}(w). 
\end{align*}
By Lemma 4.1 of \cite{Ozawa18}, these matrix series absolutely converge entry-wise in $z\in\bar{\Delta}_{e^{\underline{\theta}^{\{1,2\}}_1},e^{\bar\theta^{\{1,2\}}_1}}$ and $w\in\bar{\Delta}_{e^{\underline{\theta}^{\{1,2\}}_2},e^{\bar\theta^{\{1,2\}}_2}}$, respectively. We call $G_1(z)$ and $G_2(w)$ the G-matrix functions generated from triplets $\{A^{\{1,2\}}_{*,-1}(z), A^{\{1,2\}}_{*,0}(z), A^{\{1,2\}}_{*,1}(z)\}$ and $\{A^{\{1,2\}}_{-1,*}(w), A^{\{1,2\}}_{0,*}(w), A^{\{1,2\}}_{1,*}(w)\}$, respectively. 
The G-matrix functions $G_1(z)$ and $G_2(w)$ satisfy the following matrix quadratic equations: 
\begin{align}
&A^{\{1,2\}}_{*,-1}(z)+A^{\{1,2\}}_{*,0}(z) G_1(z) +A^{\{1,2\}}_{*,1}(z) G_1(z)^2 = G_1(z), \label{eq:G1function_equation} \\
&A^{\{1,2\}}_{-1,*}(w)+A^{\{1,2\}}_{0,*}(w) G_2(w) +A^{\{1,2\}}_{1,*}(w) G_2(w)^2 = G_2(w). \label{eq:G2function_equation}
\end{align}
From Proposition 2.5 of \cite{Ozawa18}, we see that, for $x\in[e^{\underline{\theta}^{\{1,2\}}_1}, e^{\bar\theta^{\{1,2\}}_1}]$, the Perron-Frobenius eigenvalue of $G_1(x)$ is given by $e^{\underline{\eta}_2(\log x)}$, i.e., $\spr(G_1(x))=e^{\underline{\eta}_2(\log x)}$, and for $y\in[e^{\underline{\theta}^{\{1,2\}}_2}, e^{\bar\theta^{\{1,2\}}_2}]$, that of $G_2(y)$ is given by $e^{\underline{\eta}_1(\log y)}$, i.e., $\spr(G_2(y))=e^{\underline{\eta}_1(\log y)}$. 
%
%
%
By Lemma 4.2 of \cite{Ozawa18}, $G_1(z)$ is entry-wise analytic in the open annular domain $\Delta_{e^{\underline{\theta}^{\{1,2\}}_1},e^{\bar\theta^{\{1,2\}}_1}}$ and $G_2(w)$ is entry-wise analytic in $\Delta_{e^{\underline{\theta}^{\{1,2\}}_2},e^{\bar\theta^{\{1,2\}}_2}}$. 
The convergence domain of $G_1(e^{\theta_1})$ is given by the open interval $[\Gamma^{\{1,2\}}]_{\{1\}}=(\underline{\theta}^{\{1,2\}}_1, \bar\theta^{\{1,2\}}_1)$ and that of $G_2(e^{\theta_2})$ by $[\Gamma^{\{1,2\}}]_{\{2\}}=(\underline{\theta}^{\{1,2\}}_2, \bar\theta^{\{1,2\}}_2)$.

%
For $i,j\in\{0,1\}$, define the following matrix-valued generating functions:
\begin{align*}
&A^{\{1\}}_{*,j}(z) = \sum_{i'\in\{-1,0,1\}} z^{i'} A^{\{1\}}_{i',j},\quad 
A^{\{2\}}_{i,*}(w) = \sum_{j'\in\{-1,0,1\}} w^{j'} A^{\{2\}}_{i,j'}, 
\end{align*}
and consider 1d-QBD-type block matrices $K^{\{1\}}(z)=(K^{\{1\}}_{x_2,x_2'}(z); x_2,x_2'\in\mathbb{Z}_+)$ whose nonzero blocks are given by the twins $\{A^{\{1\}}_{*,0}(z), A^{\{1\}}_{*,1}(z)\}$ and the triplet $\{A^{\{1,2\}}_{*,-1}(z), A^{\{1,2\}}_{*,0}(z), A^{\{1,2\}}_{*,1}(z)\}$ and $K^{\{2\}}(w)=(K^{\{2\}}_{x_1,x_1'}(w); x_1,x_1'\in\mathbb{Z}_+)$ whose nonzero blocks are given by the twins $\{A^{\{2\}}_{0,*}(w), A^{\{2\}}_{1,*}(w)\}$ and the triplet $\{A^{\{1,2\}}_{-1,*}(w), A^{\{1,2\}}_{0,*}(w), A^{\{1,2\}}_{1,*}(w)\}$. 
For the 1d-QBD-type block matrices $K^{\{1\}}(z)$ and $K^{\{2\}}(z)$, the hitting measures from level $k$ to level $0$ are given by $G_1(z)^k$ and $G_2(w)^k$, respectively. Hence, we have 
\begin{align}
\Phi^{\{1\}}_{(0,k),(*,0)}(z) = G_1(z)^k \Phi^{\{1\}}_{(0,0),(*,0)}(z),\quad 
\Phi^{\{2\}}_{(k,0),(0,*)}(w) = G_2(w)^k \Phi^{\{2\}}_{(0,0),(0,*)}(w). 
\end{align}
Hereafter, we denote $\Phi^{\{1\}}_{(0,0),(*,0)}(z)$ and $\Phi^{\{2\}}_{(0,0),(0,*)}(w)$ by $\Phi^{\{1\}}_*(z)$ and $\Phi^{\{2\}}_*(w)$, respectively. By \eqref{eq:Gamma1_eq1}   and \eqref{eq:Gamma2_eq1}, the convergence domain of $\Phi^{\{1\}}_*(e^{\theta_1})$ is given by $\Gamma_0^{\{1\}}$ and that of $\Phi^{\{2\}}_*(e^{\theta_2})$ by $\Gamma_0^{\{2\}}$. 
By \eqref{eq:Phi1and2_eq1}, we have 
\begin{align}
&\Phi^{\{1\}}_{(x_1,x_2),(*,0)}(z) = z^{x_1} G_1(z)^{x_2} \Phi^{\{1\}}_*(z),\ (x_1,x_2)\in\mathbb{Z}\times\mathbb{Z}_+, \label{eq:Phi1_eq1} \\
&\Phi^{\{2\}}_{(x_1,x_2),(0,*)}(w) = w^{x_2} G_2(z)^{x_1} \Phi^{\{2\}}_*(w),\ (x_1,x_2)\in\mathbb{Z}_+\times\mathbb{Z}.  \label{eq:Phi2_eq1} 
\end{align}
In terms of $G_1(z)$ and $G_2(z)$, $\Phi^{\{1\}}_*(z)$ and $\Phi^{\{2\}}_*(w)$ are represented as 
\begin{align}
&\Phi^{\{1\}}_*(z) 
= \sum_{n=0}^\infty \left( A^{\{1\}}_{*,0}(z)+A^{\{1\}}_{*,1}(z)G_1(z) \right)^n 
= \left(I- A^{\{1\}}_{*,0}(z)-A^{\{1\}}_{*,1}(z)G_1(z) \right)^{-1}, \label{eq:Phi1z_series}\\
&\Phi^{\{2\}}_*(w) 
= \sum_{n=0}^\infty \left( A^{\{2\}}_{0,*}(w)+A^{\{2\}}_{1,*}(w)G_2(w) \right)^n
= \left(I- A^{\{2\}}_{0,*}(w)-A^{\{2\}}_{1,*}(w)G_2(w) \right)^{-1}, 
\end{align}
and this implies that
\begin{align}
&\Gamma_0^{\{1\}} = \{\theta_1\in\mathbb{R}; \spr(A^{\{1\}}_{*,0}(e^{\theta_1})+A^{\{1\}}_{*,1}(e^{\theta_1})G_1(e^{\theta_1}))<1\}, \\
&\Gamma_0^{\{2\}} = \{\theta_2\in\mathbb{R}; \spr(A^{\{2\}}_{0,*}(e^{\theta_2})+A^{\{2\}}_{1,*}(e^{\theta_2})G_2(e^{\theta_2}))<1\}.
\end{align}
%

%
\subsection{Convergence domains of the generating functions}

By the definition of $\hat\bnu^{\{1\}}_{(x_1,x_2)}$ and $\hat\bnu^{\{2\}}_{(x_1,x_2)}$, we have
\begin{align*}
\hat\bnu_{(*,0)}(z) = \sum_{x_1=-\infty}^\infty z^{x_1} \hat\bnu^{\{1\}}_{(x_1,0)},\quad 
\hat\bnu_{(0,*)}(w) = \sum_{x_2=-\infty}^\infty w^{x_2} \hat\bnu^{\{2\}}_{(0,x_2)}. 
\end{align*}
Hence, from \eqref{eq:hatnux1}, \eqref{eq:Phi1_eq1} and \eqref{eq:G1function_equation}, we obtain 
\begin{align}
\hat\bnu_{(*,0)}(z)
&= -(\hat\bnu_{(1,0)} A^{\{1\}}_{-1,0}+ \hat\bnu_{(0,1)} A^{\{2\}}_{0,-1} + \hat\bnu_{(1,1)} A^{\{1,2\}}_{-1,-1}) \Phi^{\{1\}}_*(z) \cr
&\quad + \sum_{k=1}^\infty \hat\bnu_{(0,k)} (\hat A^{\{2\}}_{*,*}(z,G_1(z))-G_1(z)) G_1(z)^{k-1} \Phi^{\{1\}}_*(z),
\label{eq:hatnux1s0} 
\end{align}
where
\[
\hat A^{\{2\}}_{*,*}(z,G_1(z)) = A^{\{2\}}_{*,-1}(z)+A^{\{2\}}_{*,0}(z) G_1(z)+A^{\{2\}}_{*,1}(z) G_1(z)^2,\quad 
A^{\{2\}}_{*,j}(z) = \sum_{i\in\{0,1\}} z^i A^{\{2\}}_{i,j},
\]
and from \eqref{eq:hatnux2}, \eqref{eq:Phi2_eq1} and \eqref{eq:G2function_equation}, 
\begin{align}
\hat\bnu_{(0,*)}(w)
&= -( \hat\bnu_{(1,0)} A^{\{1\}}_{-1,0} +\hat\bnu_{(0,1)} A^{\{2\}}_{0,-1}+ \hat\bnu_{(1,1)} A^{\{1,2\}}_{-1,-1} )\Phi^{\{2\}}_*(w) \cr
&\quad + \sum_{k=1}^\infty \hat\bnu_{(k,0)} (\hat A^{\{1\}}_{*,*}(G_2(w),w)-G_2(w)) G_2(w)^{k-1} \Phi^{\{2\}}_*(w),
\label{eq:hatnux2s0} 
\end{align}
where
\[
\hat A^{\{1\}}_{*,*}(G_2(w),w) = A^{\{1\}}_{-1,*}(w)+A^{\{1\}}_{0,*}(w) G_2(w)+A^{\{1\}}_{1,*}(w) G_2(w)^2,\quad 
A^{\{1\}}_{i,*}(w) = \sum_{j\in\{0,1\}} w^j A^{\{1\}}_{i,j}.
\]
By \eqref{eq:hatnux1s0} and \eqref{eq:hatnux2s0}, the convergence domains $\calD^{\{1\}}$ and $\calD^{\{2\}}$  satisfy the following. 
%
\begin{proposition} \label{pr:calD1D2_ineq1}
We have 
\begin{align}
&[\calD^{\{2\}}\cap\Gamma^{\{1,2\}}\cap\Gamma^{\{1\}} ]_{\{1\}}^{ex} \subset \calD_0^{\{1\}}, \label{eq:calD1_ineq1}\\
&[\calD^{\{1\}}\cap\Gamma^{\{1,2\}}\cap\Gamma^{\{2\}} ]_{\{2\}}^{ex} \subset \calD_0^{\{2\}}. \label{eq:calD2_ineq1}
\end{align}
\end{proposition}
\begin{proof}
For $s_1\in(\underline{\theta}^{\{1,2\}}_1,\bar\theta^{\{1,2\}}_1)=[\Gamma^{\{1,2\}}]_{\{1\}}$, $G_1(e^{s_1})$ is entry-wise finite and we have $\cp(G_1(e^{s_1}))=e^{-\underline{\eta}_2(s_1)}$ (see Proposition 5.1 of \cite{Ozawa18}). Recall that the lower asymptotic decay rate $\underline{\xi}_{(0,1)}$ satisfies 
\[
\underline{\xi}_{(0,1)} = \sup\{\theta_2\in\mathbb{R}; \mbox{$(\theta_1,\theta_2)\in\calD^{\{2\}}$ for some $\theta_1\in\mathbb{R}$} \}.
\]
Set $z=e^{s_1}$ for $s_1\in[\Gamma^{\{1,2\}}]_{\{1\}}$. By the Cauchy-Hadamard theorem, if $\underline{\eta}_2(s_1)<\underline{\xi}_{(0,1)}$, the matrix-valued series in \eqref{eq:hatnux1s0} converges entry-wise. The condition that  $s_1\in[\Gamma^{\{1,2\}}]_{\{1\}}$ and $\underline{\eta}_2(s_1)<\underline{\xi}_{(0,1)} $ is equivalent to one that $s_1\in[\calD^{\{2\}}\cap\Gamma^{\{1,2\}} ]_{\{1\}}$ (see Fig.\ \ref{fig:fig31}). 
Furthermore, if $s_1\in\Gamma_0^{\{1\}}=[\Gamma^{\{1\}}]_{\{1\}}$, $\Phi^{\{1\}}_*(e^{s_1})$ is entry-wise finite. Hence, if $s_1\in[\calD^{\{2\}}\cap\Gamma^{\{1,2\}}\cap\Gamma^{\{1\}} ]_{\{1\}}$, $\hat\bnu_{(*,0)}(e^{s_1})$ is entry-wise finite and we obtain \eqref{eq:calD1_ineq1}. 
Inclusion relation \eqref{eq:calD2_ineq1} can analogously be obtained. 
\end{proof}
%
\begin{figure}[t]
\begin{center}
\includegraphics[width=55mm,trim=0 0 0 0]{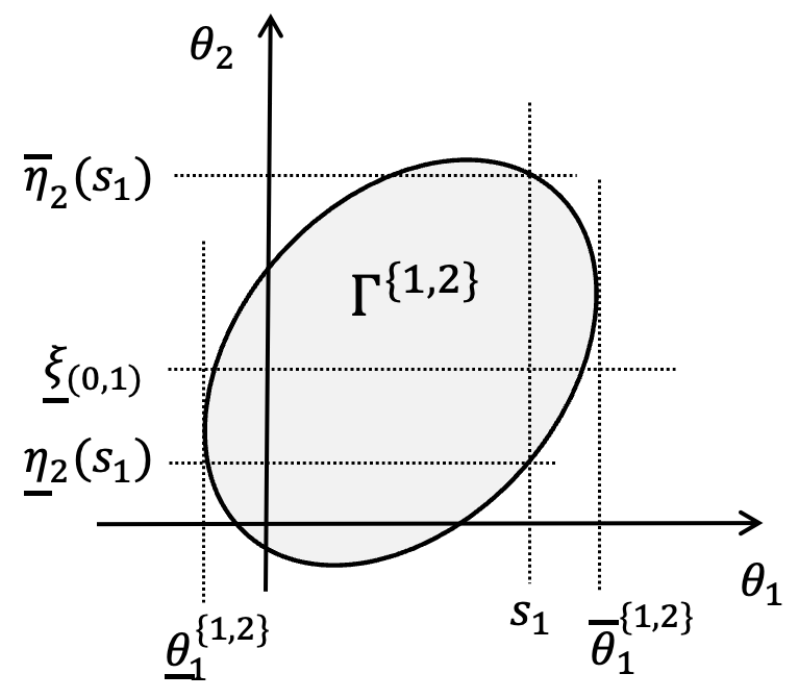} 
\caption{Convergence domain}
\label{fig:fig31}
\end{center}
\end{figure}
%

%
Next, we consider the domain $\calD$. By the definition of $\hat\bnu^{\{1,2\}}_{(x_1,x_2)}$, we have 
\[
\hat\bnu_*(z,w) = \sum_{x_1=-\infty}^\infty \sum_{x_2=-\infty}^\infty z^{x_1} w^{x_2} \hat\bnu^{\{1,2\}}_{(x_1,x_2)}. 
\]
Hence, from \eqref{eq:hatnux12} and \eqref{eq:Phi12_eq1}, we obtain
\begin{align}
\hat\bnu_*(z,w)
&= - (\hat\bnu_{(1,0)} A^{\{1\}}_{-1,0}+ \hat\bnu_{(0,1)} A^{\{2\}}_{0,-1}+ \hat\bnu_{(1,1)} A^{\{2\}}_{-1,-1}) \Phi^{\{1,2\}}_*(z,w) \cr
&\quad + \hat\bnu_{(*,0)}(z) (A^{\{1\}}_{*,*}(z,w)-A^{\{1,2\}}_{*,*}(z,w)) \Phi^{\{1,2\}}_*(z,w) \cr
&\quad + \hat\bnu_{(0,*)}(w) (A^{\{2\}}_{*,*}(z,w)-A^{\{1,2\}}_{*,*}(z,w)) \Phi^{\{1,2\}}_*(z,w), 
\label{eq:hatnux12s}
\end{align}
where, for $\alpha\in\{\{1\},\{2\},\{1,2\}\}$, we set $A^\alpha_{*,*}(z,w)=\sum_{i_1,i_2\in\{-1,0,1\}} z^{i_1} w^{i_2} A^\alpha_{i_1,i_2}$ and denote the matrix-valued generating function $\Phi^{\{1,2\}}_{(0,0),(*,*)}(z,w)$ by $\Phi^{\{1,2\}}_*(z,w)$. Since the convergence domain of $\Phi^{\{1,2\}}_*(e^{\theta_1},e^{\theta_2})$ is given by $\Gamma^{\{1,2\}}$, we immediately obtain the following.
\begin{proposition} \label{pr:calD_ineq1}
We have 
\begin{equation}
[\calD^{\{1\}}\cap\calD^{\{2\}}\cap\Gamma^{\{1,2\}}]^{ex} \subset \calD. \label{eq:calD_ineq1}
\end{equation}
\end{proposition}

%
%

\section{Asymptotics of the occupation measure} \label{sec:asymptotics_o}

%
\subsection{Identifying the convergence domains}

Recall that the convergence domains of $\Phi^{\{1\}}_*(e^{\theta_1})$, $\Phi^{\{2\}}_*(e^{\theta_2})$ and $\Phi^{\{1,2\}}_*(e^{\theta_1},e^{\theta_2})$ are given by the regions $\Gamma_0^{\{1\}}$, $\Gamma_0^{\{2\}}$ and $\Gamma^{\{1,2\}}$, respectively. Hence, in order to make it possible that $\Phi^{\{1\}}_*(z)$, $\Phi^{\{2\}}_*(w)$ and $\Phi^{\{1,2\}}_*(z,w)$ are simultaneously finite, we assume the following condition throughout the paper. 
\begin{assumption} \label{as:Gamma1212_notempty}
$\Gamma^{\{1\}}\cap\Gamma^{\{2\}}\cap\Gamma^{\{1,2\}} \ne \emptyset$.
\end{assumption} 

Furthermore, in order to make it possible that the matrix-valued series in \eqref{eq:hatnux1s0} and \eqref{eq:hatnux2s0} as well as $\Phi^{\{1\}}_*(z)$ and $\Phi^{\{2\}}_*(w)$ are simultaneously finite, we assume the following condition throughout the paper. 
\begin{assumption} \label{as:Gamma12D1D2_notempty}
$\Gamma^{\{1\}}\cap\calD^{\{2\}}\cap\Gamma^{\{1,2\}} \ne \emptyset$ and $\calD^{\{1\}}\cap\Gamma^{\{2\}}\cap\Gamma^{\{1,2\}} \ne \emptyset$.
\end{assumption} 

\begin{remark}
The conditions in Assumption \ref{as:Gamma12D1D2_notempty} may be derived from other assumptions. We assume it  to make our discussion simple. 
The former condition ensures that the right hand side of \eqref{eq:hatnux1s0} is finite for some positive point $z$, and the latter one that the right hand side of \eqref{eq:hatnux2s0} is finite for some positive point $w$. 
\end{remark}

For $i\in\{1,2\}$, let $\bar\theta^{\{i\}}_i$ and $\underline{\theta}^{\{i\}}_i$ be the extreme points of $\Gamma_0^{\{i\}}$ defined as
\[
\bar\theta^{\{i\}}_i = \sup_{\theta_i\in\Gamma_0^{\{i\}}} \theta_i,\quad 
\underline{\theta}^{\{i\}}_i = \inf_{\theta_i\in\Gamma_0^{\{i\}}} \theta_i.
\]
By the definition of $\Gamma_0^{\{1\}}$, $\Gamma_0^{\{2\}}$ and $\Gamma^{\{1,2\}}$, we have 
\begin{equation}
\Gamma_0^{\{1\}}\subset[\Gamma^{\{1,2\}}]_{\{1\}},\quad
\Gamma_0^{\{2\}}\subset[\Gamma^{\{1,2\}}]_{\{2\}}, 
\end{equation}
and this leads us to
\begin{equation}
\underline{\theta}^{\{1,2\}}_1\le \underline{\theta}^{\{1\}}_1\le \bar\theta^{\{1\}}_1\le  \bar\theta^{\{1,2\}}_1,\quad 
\underline{\theta}^{\{1,2\}}_2\le \underline{\theta}^{\{2\}}_2\le \bar\theta^{\{2\}}_2\le  \bar\theta^{\{1,2\}}_2.
\end{equation}
We introduced the following optimization problem in Section \ref{sec:intro}.
\begin{align}
\begin{array}{ll}
\mbox{Maximize} & s_1+s_2, \cr
\mbox{subject to} & s_1 \le \sup\!\left\{\theta_1\in\mathbb{R}; (\theta_1,\theta_2)\in \Gamma^{\{1\}}\cap\Lambda^{\{2\}}(s_2)\cap\Gamma^{\{1,2\}}\ \mbox{for some $\theta_2\in\mathbb{R}$} \right\}, \cr
&  s_2 \le \sup\!\left\{\theta_2\in\mathbb{R}; (\theta_1,\theta_2)\in\Lambda^{\{1\}}(s_1)\cap\Gamma^{\{2\}}\cap\Gamma^{\{1,2\}}\ \mbox{for some $\theta_1\in\mathbb{R}$} \right\}.
\end{array}
\end{align}
The restriction of the optimization problem corresponds to
\begin{align}
\Lambda_0(s_1)\subset [\Gamma^{\{1\}}\cap\Lambda^{\{2\}}(s_2)\cap\Gamma^{\{1,2\}}]^{ex}_{\{1\}}, \\
\Lambda_0(s_2)\subset [\Lambda^{\{1\}}(s_1)\cap\Gamma^{\{2\}}\cap\Gamma^{\{1,2\}}]^{ex}_{\{2\}}, 
\end{align}
where for $s\in\mathbb{R}$, $\Lambda_0(s)=(-\infty, s)$. 
Under Assumption \ref{as:Gamma1212_notempty}, this optimization problem is feasible. Since the regions $\Gamma^{\{1\}}$, $\Gamma^{\{2\}}$ and $\Gamma^{\{1,2\}}$ are convex sets, it has a unique optimal solution $(s_1^*,s_2^*)$, which satisfy 
\begin{align}
\Lambda_0(s_1^*)= [\Gamma^{\{1\}}\cap\Lambda^{\{2\}}(s_2^*)\cap\Gamma^{\{1,2\}}]^{ex}_{\{1\}}, \label{eq:s1s_eq1}\\
\Lambda_0(s_2^*)= [\Lambda^{\{1\}}(s_1^*)\cap\Gamma^{\{2\}}\cap\Gamma^{\{1,2\}}]^{ex}_{\{2\}},  \label{eq:s2s_eq1}
\end{align}
and the optimal solution is given as
\begin{align}
s_1^* = \sup\!\left\{\theta_1\in\mathbb{R}; (\theta_1,\theta_2)\in\Gamma^{\{1\}}\cap[\Gamma^{\{2\}}]^{ex}\cap\Gamma^{\{1,2\}}\ \mbox{for some $\theta_2\in\mathbb{R}$} \right\}, \label{eq:s1s_eq2}\\
s_2^* = \sup\!\left\{\theta_2\in\mathbb{R}; (\theta_1,\theta_2)\in[\Gamma^{\{1\}}]^{ex}\cap\Gamma^{\{2\}}\cap\Gamma^{\{1,2\}}\ \mbox{for some $\theta_1\in\mathbb{R}$} \right\}.  \label{eq:s2s_eq2}
\end{align}
%
By Proposition \ref{pr:calD1D2_ineq1}, we obtain the following.  
\begin{proposition} \label{pr:calD1D2_ineq2} 
$\Lambda_0(s_1^*)\subset\calD_0^{\{1\}}$ and $\Lambda_0(s_2^*)\subset\calD_0^{\{2\}}$.
\end{proposition}
%
\begin{figure}[t]
\begin{center}
\includegraphics[width=55mm,trim=0 0 0 0]{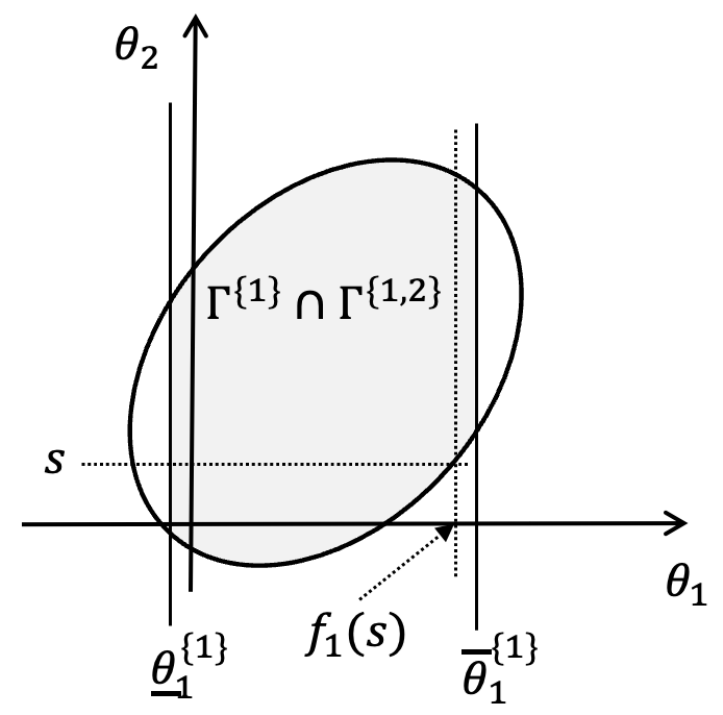} 
\caption{Function $f_1(s)$}
\label{fig:fig41}
\end{center}
\end{figure}
%
\begin{proof}
Define functions $f_1(s)$ and $f_2(s)$ as 
\begin{align*}
&f_1(s)= \sup\left\{\theta_1\in\mathbb{R}; (\theta_1,\theta_2)\in \Gamma^{\{1\}}\cap\Lambda^{\{2\}}(s)\cap\Gamma^{\{1,2\}}\ \mbox{for some $\theta_2\in\mathbb{R}$} \right\}, \\
&f_2(s)= \sup\left\{\theta_2\in\mathbb{R}; (\theta_1,\theta_2)\in\Lambda^{\{1\}}(s)\cap\Gamma^{\{2\}}\cap\Gamma^{\{1,2\}}\ \mbox{for some $\theta_1\in\mathbb{R}$} \right\}, 
\end{align*}
and set $f_{21}(s)=f_1(f_2(s))$. 
For some $s_0$ in the closure of $[\Gamma^{\{1\}}\cap\Gamma^{\{1,2\}}]_{\{2\}}$, $f_1(s)$ is increasing in $s<s_0$ and $f_1(s)=\bar\theta^{\{1\}}_1$ for $s\ge s_0$ (see Fig.\ \ref{fig:fig41}). For some $s_0'$ in the closure of $[\Gamma^{\{2\}}\cap\Gamma^{\{1,2\}}]_{\{1\}}$, $f_2(s)$ is increasing in $s<s_0'$ and $f_2(s)=\bar\theta^{\{2\}}_2$ for $s\ge s_0'$. 
Furthermore, since the closure of $\Gamma^{\{1,2\}}$ is a convex set, $f_1(s)$ is concave in $s\in[\Gamma^{\{1\}}\cap\Gamma^{\{1,2\}}]_{\{2\}}$ (see Fig.\ \ref{fig:fig41}) and $f_2(s)$ is concave in $s\in[\Gamma^{\{2\}}\cap\Gamma^{\{1,2\}}]_{\{1\}}$.
We have $s_1^*=f_1(s_2^*)$, $s_2^*=f_2(s_1^*)$ and $s_1^*=f_{21}(s_1^*)$. Hence, $f_{21}(s)$ is increasing in $s<s_1^*$ and $f_{21}(s)=s_1^*$ for $s\ge s_1^*$. Furthermore, since $f_{21}(s)$ is concave in $s\in[\Gamma^{\{2\}}\cap\Gamma^{\{1,2\}}]_{\{1\}}$, we see that if $s<s_1^*$, then $f_{21}(s)>s$. 
From the proof of Proposition \ref{pr:calD1D2_ineq1}, we also see that if the matrix-valued generating function $\hat\bnu_{(*,0)}(e^{\theta_1})$ converges for $\theta_1<s$, then $\hat\bnu_{(0,*)}(e^{\theta_2})$ converges for $\theta_2<f_2(s)$ and $\hat\bnu_{(*,0)}(e^{\theta_1})$ does for $\theta_1<f_{21}(s)$. 

Under Assumption \ref{as:Gamma12D1D2_notempty}, we have $\underline{\xi}_{(1,0)}\in[\Gamma^{\{2\}}\cap\Gamma^{\{1,2\}}]_{\{1\}}$. 
Suppose $\underline{\xi}_{(1,0)}<s_1^*$. Then, we have $f_{21}(\underline{\xi}_{(1,0)})>\underline{\xi}_{(1,0)}$ and $\hat\bnu_{(*,0)}(e^{\theta_1})$ converges for $\theta_1<f_{21}(\underline{\xi}_{(1,0)})$. This contradicts the definition of $\underline{\xi}_{(1,0)}$. Hence, we must have $\underline{\xi}_{(1,0)}\ge s_1^*$, i.e., $\Lambda_0(s_1^*)\subset\calD_0^{\{1\}}$. 
Another inclusion relation can analogously be obtained. 
\end{proof}

By Proposition \ref{pr:calD1D2_ineq2}, $\hat\bnu_{(*,0)}(z)$ is analytic in $z\in\Delta_{e^{s_1^*}}$ and  $\hat\bnu_{(0,*)}(w)$ is analytic in $w\in\Delta_{e^{s_2^*}}$. Hence, if $z=e^{s_1^*}$ and $w=e^{s_2^*}$ are singular points of $\hat\bnu_{(*,0)}(z)$ and $\hat\bnu_{(0,*)}(w)$, respectively, we obtain the following.
\begin{theorem} \label{pr:calD1D2_eq}
$\calD_0^{\{1\}}=\Lambda_0(s_1^*)$ and $\calD_0^{\{2\}}=\Lambda_0(s_2^*)$.
\end{theorem}

Before proving the theorem, we define real values $\bar\theta_1^*$ and $\bar\theta_2^*$ as 
\begin{align*}
&\bar\theta_1^* = \sup\{\theta_1\in\mathbb{R}; (\theta_1,\theta_2)\in [\Gamma^{\{2\}}]^{ex}\cap\Gamma^{\{1,2\}}\ \mbox{for some $\theta_2\in\mathbb{R}$} \}, \\
&\bar\theta_2^* = \sup\{\theta_2\in\mathbb{R}; (\theta_1,\theta_2)\in [\Gamma^{\{1\}}]^{ex}\cap\Gamma^{\{1,2\}}\ \mbox{for some $\theta_1\in\mathbb{R}$} \}. 
\end{align*}
%
\begin{figure}[t]
\begin{center}
\includegraphics[width=55mm,trim=0 0 0 0]{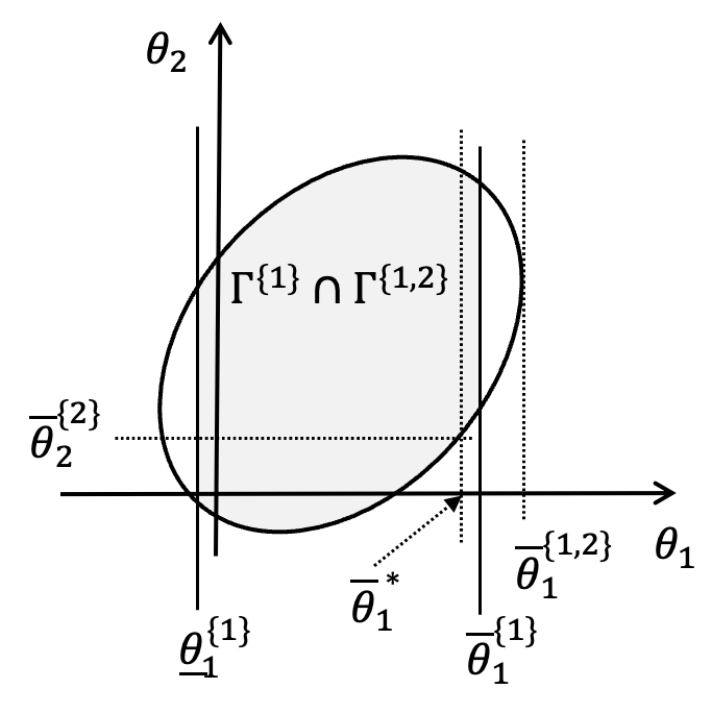} 
\caption{An example of case (C4: $s_1^*=\bar\theta_1^*$, $s_2^*=\bar\theta_2^{\{2\}}$)}
\label{fig:fig42}
\end{center}
\end{figure}
%
\begin{proof}[Proof of Theorem \ref{pr:calD1D2_eq}]
Since this theorem can be proved in a manner similar to that used in Ozawa and Kobayashi \cite{Ozawa18} and Ozawa \cite{Ozawa23}, we give only an outline. 
Denote by $\ba(z)$ the matrix-valued series appearing on the right hand side of \eqref{eq:hatnux1s0}, i.e., 
\begin{equation}
\ba(z) =  \sum_{k=1}^\infty \hat\bnu_{(0,k)} (\hat A^{\{2\}}_{*,*}(z,G_1(z))-G_1(z)) G_1(z)^{k-1}, 
\label{eq:a_def1}
\end{equation}
and recall that $\Phi^{\{1\}}_*(z)$ is represented as a matrix-valued series by \eqref{eq:Phi1z_series}. 
According to \eqref{eq:s1s_eq1}, we consider the following cases (see Fig.\ \ref{fig:fig42}):
\begin{align*}
&\mbox{(C1)}\ s_1^*=\bar\theta_1^{\{1,2\}}, \cr
&\mbox{(C2)}\ s_1^*=\bar\theta_1^{\{1\}}<\bar\theta_1^*, \cr
&\mbox{(C3)}\ s_1^*=\bar\theta_1^{\{1\}}=\bar\theta_1^*<\bar\theta_1^{\{1,2\}}, \cr
&\mbox{(C4)}\ s_1^*=\bar\theta_1^*<\bar\theta_1^{\{1\}},
\end{align*}
where we always have $\bar\theta_1^{\{1\}}\le\bar\theta_1^{\{1,2\}}$ and $\bar\theta_1^*\le\bar\theta_1^{\{1,2\}}$. 

First, we consider the case of C1. The point $z=e^{s_1^*}=e^{\bar\theta_1^{\{1,2\}}}$ is a brunch point of $G_1(z)$ with order one (see Lemma 4.7 of \cite{Ozawa18} and Corollary 3.1 and Lemma 3.3 of \cite{Ozawa23}). 
Hence, by \eqref{eq:a_def1} and \eqref{eq:Phi1z_series}, $z=e^{s_1^*}$ is also a brunch point of $\ba(z)$ and $\Phi^{\{1\}}_*(z)$, and we see by \eqref{eq:hatnux1s0} that it is a brunch point of $\hat\bnu_{(*,0)}(z)$. 

Second, we consider the case of C2. By \eqref{eq:Phi1z_series}, we see that $z=e^{s_1^*}=e^{\bar\theta_1^{\{1\}}}$ is a pole of $\Phi^{\{1\}}_*(z)$ with order one (see Propositions 5.3 and 5.4 of \cite{Ozawa18}). 
On the other hand, since $s_1^*<\bar\theta_1^*$, $\ba(z)$ converges at $z=e^{s_1^*}$ and it is analytic at the same point (see Proposition 5.1 of \cite{Ozawa18}). 
Furthermore, we see that $\ba(e^{s_1^*})$ is nonzero (see Lemma 5.3 of \cite{Ozawa18}). Hence, the point $z=e^{s_1^*}$ is a pole of $\hat\bnu_{(*,0)}(z)$ with order one. 

Third, we consider the case of C3. In this case, the point $z=e^{s_1^*}=e^{\bar\theta_1^{\{1\}}}=e^{\bar\theta_1^*}$ is also a pole of $\Phi^{\{1\}}_*(z)$ with order one (see discussion above). 
Since we have $s_2^*=\bar\theta^{\{2\}}_2=\underline{\eta}_2(s_1^*)<\bar\theta_2^*$, the point $w=e^{\underline{\eta}_2(s_1^*)}$ is a pole of $\hat\bnu_{(0,*)}(w)$, where $\cp(G(e^{s_1^*}))=e^{-\underline{\eta}_2(s_1^*)}$ (see the proof of Lemma 5.5 (2) in \cite{Ozawa18}). 
Hence, by \eqref{eq:a_def1}, we see that $z=e^{s_1^*}$ is a pole of $\ba(z)$ with order one, and it is a pole of $\hat\bnu_{(*,0)}(z)$ with order two (see Lemma 5.5 (2) of \cite{Ozawa18}). 

Finally, we consider the case of C4. Since $s_1^*<\bar\theta_1^{\{1\}}$, $\Phi^{\{1\}}_*(z)$ converges at $z=e^{s_1^*}$. 
Because of the same reason as C3, the point $z=e^{s_1^*}$ is a pole of $\ba(z)$ with order one, and it is a pole of $\hat\bnu_{(*,0)}(z)$ with order one. 

As a result, we see that the point $z=e^{s_1^*}$ is a singular point of $\hat\bnu_{(*,0)}(z)$ in all the cases. 
Analogously, we can see that the point $w=e^{s_2^*}$ is always a singular point of $\hat\bnu_{(0,*)}(w)$. 
\end{proof}

Let $\Gamma$ be a region defined as 
\[
\Gamma = [\Gamma^{\{1\}}]^{ex} \cap [\Gamma^{\{2\}}]^{ex} \cap \Gamma^{\{1,2\}}.
\]
Under Assumption \ref{as:Gamma1212_notempty}, the convergence domains $\calD_0^{\{1\}}$ and $\calD_0^{\{2\}}$ are represented as
\begin{equation}
\calD_0^{\{1\}}=(-\infty,s_1^*)=[\Gamma]_{\{1\}}^{ex},\quad 
\calD_0^{\{2\}}=(-\infty,s_2^*)=[\Gamma]_{\{2\}}^{ex}.
\label{eq:calD10D20_eq1}
\end{equation}
The convergence domain $\calD$ is given as follows.
\begin{theorem} \label{pr:calD_eq}
We have 
\begin{equation}
\calD=[\calD^{\{1\}}\cap\calD^{\{2\}}\cap\Gamma^{\{1,2\}}]^{ex}=[\Gamma]^{ex}.
\label{eq:calD_eq1}
\end{equation}
\end{theorem}

\begin{proof}
From Proposition \ref{le:domain_upper}, we obtain $\calD \subset [\Gamma^{\{1,2\}}]^{ex}$. Obviously, we have $\calD\subset\calD^{\{1\}}\cap\calD^{\{2\}}$. Hence, we obtain 
\[
\calD \subset \calD^{\{1\}}\cap\calD^{\{2\}}\cap[\Gamma^{\{1,2\}}]^{ex}.
\]
Suppose there exists a point $(s_1,s_2)\in\mathbb{R}^2$ such that 
\[
(s_1,s_2) \in (\calD^{\{1\}}\cap\calD^{\{2\}}\cap[\Gamma^{\{1,2\}}]^{ex})\setminus [\calD^{\{1\}}\cap\calD^{\{2\}}\cap\Gamma^{\{1,2\}}]^{ex}, 
\]
where $s_1<s_1^*$ and $s_2<s_2^*$. Then, for every $(s_1',s_2')\in\mathbb{R}^2$ such that $s_1\le s_1'<s_1^*$ and $s_2\le s_2'<s_2^*$, we have $(s_1',s_2')\notin\Gamma^{\{1,2\}}$. By \eqref{eq:s1s_eq1}, if $s_2<s_2^*$, this implies $s_1^*\le s_1$; By \eqref{eq:s2s_eq1}, if $s_1<s_1^*$, it implies $s_2^*\le s_2$. 
Hence, we get a contradiction. As a result, we obtain 
\[
\calD^{\{1\}}\cap\calD^{\{2\}}\cap[\Gamma^{\{1,2\}}]^{ex} = [\calD^{\{1\}}\cap\calD^{\{2\}}\cap\Gamma^{\{1,2\}}]^{ex}, 
\]
and, from \eqref{eq:calD_ineq1},  the first equality of \eqref{eq:calD_eq1} is derived.
By \eqref{eq:calD10D20_eq1}, the second equality of \eqref{eq:calD_eq1} is obvious.
\end{proof}

%
\subsection{Asymptotic decay rates and functions}

With respect to the asymptotic decay rates and functions of the occupation measure, we can also derive the same results as those obtained in Ozawa and Kobayashi \cite{Ozawa18} and Ozawa \cite{Ozawa22,Ozawa23}. In this subsection, we summarize them. 

We note that the singular points of the matrix-valued generating functions $\hat\bnu_{(*,0)}(z)$ and $\hat\bnu_{(0,*)}(w)$ that determine the asymptotic nature of the occupation measure are poles and branch points with a finite order. Furthermore, such a singular point for $\hat\bnu_{(*,0)}(z)$ is common to all the entries of $\hat\bnu_{(*,0)}(z)$. Hence, the upper asymptotic decay rates in directions $(1,0)$ and $(0,1)$ are identical to the corresponding lower asymptotic decay rates and they do not depend on the phase state, i.e., for every $j_0\in S_0$, $\bar\xi_{(1,0),j_0}=\underline{\xi}_{(1,0),j_0}=\xi_{(1,0),j_0}=\xi_{(1,0)}$ and $\bar\xi_{(0,1),j_0}=\underline{\xi}_{(0,1),j_0}=\xi_{(0,1),j_0}=\xi_{(0,1)}$. 
For other directions, the same result also holds. 
From Theorem \ref{pr:calD1D2_eq}, the asymptotic decay rates $\xi_{(1,0)}$ and $\xi_{(0,1)}$ is given as follows. 
\begin{theorem} \label{th:xi10xi01_eq}
We have $\xi_{(1,0)}=s_1^*$ and $\xi_{(0,1)}=s_2^*$. 
\end{theorem}

In a manner similar to that used in \cite{Ozawa22,Ozawa23}, the asymptotic decay rates $\xi_{\bc}$ is given as follows (see Section 4.1 of \cite{Ozawa22} and Theorem 2.2 of \cite{Ozawa23}). 
\begin{theorem} \label{th:xic_eq}
For $\bc\in\mathbb{N}^2$, we have 
\begin{align}
\xi_{\bc} 
&= \min\!\left\{ \sup\{\langle \bc,\btheta \rangle; \btheta\in \calD^{\{1\}}\cap\Gamma^{\{1,2\}}\},\ \sup\{\langle \bc,\btheta \rangle; \btheta\in \calD^{\{2\}}\cap\Gamma^{\{1,2\}}\} \right\},
\end{align}
where $[\calD^{\{1\}}]_{\{1\}}=\calD_0^{\{1\}}=[\Gamma]^{ex}_{\{1\}}$ and $[\calD^{\{2\}}]_{\{2\}}=\calD_0^{\{2\}}=[\Gamma]^{ex}_{\{2\}}$.
\end{theorem}
\begin{remark}
We can obtain a more specific expression for the asymptotic decay rate $\xi_{\bc}$, see Theorem 2.1 of \cite{Ozawa23}. 
\end{remark}

For $\bc\in\mathbb{Z}^2\setminus\{(0,0)\}$, let $h_{\bc}(n)$ be the asymptotic decay function of the occupation measure satisfying, for some nonzero matrix $\bg_{\bc}$, 
\[
\lim_{n\to\infty} \frac{\hat\bnu_{n \bc}}{h_{\bc}(n)} = \bg_{\bc}. 
\]
In a manner similar to that used in the proof of Theorem 2.1 of \cite{Ozawa18}, we obtain the following (also see the proof of Theorem \ref{pr:calD1D2_eq}). 
\begin{theorem} \label{th:h10h01_eq}
If $\bc=(1,0)$ or $\bc=(0,1)$, then 
\begin{align}
&h_{(1,0)}(n) = \left\{ \begin{array}{ll} 
n^{-\frac{l}{2}} e^{-s_1^* n}, & \mbox{if $s_1^*=\bar\theta_1^{\{1,2\}}$, for some $l\in\{0,1,3,5,...\}$}, \cr
n e^{-s_1^* n}, & \mbox{if $s_1^*=\bar\theta_1^{\{1\}}=\bar\theta_1^*<\bar\theta_1^{\{1,2\}}$}, \cr
e^{-s_1^* n}, & \mbox{otherwise},
\end{array} \right.
\end{align}
\begin{align}
&h_{(0,1)}(n) = \left\{ \begin{array}{ll} 
n^{-\frac{l}{2}} e^{-s_2^* n}, & \mbox{if $s_2^*=\bar\theta_2^{\{1,2\}}$, for some $l\in\{0,1,3,5,...\}$}, \cr
n e^{-s_2^* n}, & \mbox{if $s_2^*=\bar\theta_2^{\{2\}}=\bar\theta_2^*<\bar\theta_2^{\{1,2\}}$}, \cr
e^{-s_2^* n}, & \mbox{otherwise}.
\end{array} \right.
\end{align}
\end{theorem}

In a manner similar to that used in the proof of Theorem 2.2 of \cite{Ozawa23}, we obtain the following. 
\begin{theorem} \label{th:h1c_eq}
For $\bc=(c_1,c_2)\in\mathbb{N}^2$, we have 
\begin{align}
&h_{\bc}(n) = \left\{ \begin{array}{ll} 
n^{-\frac{1}{2}} e^{-\xi_{\bc} n}, & \mbox{if $\bar\eta_2(s_1^*)<s_2^*$, $\bar\eta_1(s_2^*)<s_1^*$ and $\bar\eta_2'(s_1^*) < -\frac{c_1}{c_2} <\bar\eta_1'(s_2^*)^{-1}$}, \cr
e^{-\xi_{\bc} n}, & \mbox{otherwise}.
\end{array} \right.
\end{align}
\end{theorem}

%
%

\section{Asymptotics of the hitting measure} \label{sec:asymptotics_h}

From \eqref{eq:hatg_definition}, we obtain
\begin{equation}
\hat\bg^\top = \hat\bt_{10}^\top \hat\Phi^\top = \hat\bt_{10}^\top \sum_{n=0}^\infty (\hat T^\top)^n.  
\label{eq:hatg_T}
\end{equation}
Hence, the hitting measure of the nonnegative matrix $T$ is the occupation measure of the transpose of $T$. We know that the potential matrix of the transpose of a nonnegative matrix is identical to the transpose of the potential matrix of the original nonnegative matrix and the convergence parameter of the transpose of a nonnegative matrix is identical to the convergence parameter of the original nonnegative matrix. This implies that the asymptotic properties of the hitting measure can be obtained by using the results for the occupation measure. 
Furthermore, the convergence domains and asymptotic decay rates of the occupation measure defined on the nonnegative matrix $T$ are given in terms of the regions $\Gamma^{\{1\}}$, $\Gamma^{\{2\}}$ and $\Gamma^{\{1,2\}}$. Hence, in order to obtain asymptotic properties of the hitting measure, it suffices to demonstrate that those regions correspond to the regions $\Gamma_R^{\{1\}}$, $\Gamma_R^{\{2\}}$ and $\Gamma_R^{\{1,2\}}$, respectively, in the case of $T^\top$.

%
Represent $T^\top$ in block form as $T^\top=\begin{pmatrix} U_{\bx,\bx'}; \bx,\bx'\in\mathbb{Z}_+^2 \end{pmatrix}$, then its nonzero blocks are given as follows:
for $\bx=(0,0)$,
\begin{align*}
U_{(0,0),\bx'} = \left\{ \begin{array}{ll} 
(A^\emptyset_{0,0})^\top, & \bx'=(0,0), \cr
(A^{\{1\}}_{-1,0})^\top, &  \bx'=(1,0), \cr
(A^{\{2\}}_{0,-1})^\top, & \bx'=(0,1), \cr
(A^{\{1,2\}}_{-1,-1})^\top, &  \bx'=(1,1), 
\end{array} \right.
\end{align*}
for $\bx=(1,0)$,
\begin{align*}
U_{(1,0),\bx'} = \left\{ \begin{array}{ll} 
(A^\emptyset_{1,0})^\top, & \bx'=(0,0), \cr
(A^{\{1\}}_{1-x_1',0})^\top, & \bx'=(x_1',0)\in\{(1,0),(2,0)\}, \cr
(A^{\{2\}}_{1,-1})^\top, & \bx'=(0,1), \cr
(A^{\{1,2\}}_{1-x_1',-1})^\top, & \bx'=(x_1',1)\in\{(2,1),(1,1)\}, 
\end{array} \right.
\end{align*}
for $\bx=(0,1)$,
\begin{align*}
U_{(0,1),\bx'} = \left\{ \begin{array}{ll} 
(A^\emptyset_{0,1})^\top, & \bx'=(0,0), \cr
(A^{\{1\}}_{-1,1})^\top, & \bx'=(1,0), \cr
(A^{\{2\}}_{0,1-x_2'})^\top, & \bx'=(0,x_2')\in\{(0,1),(0,2)\}, \cr
(A^{\{1,2\}}_{-1,1-x_2'})^\top, & \bx'=(1,x_2)\in\{(1,2),(1,1)\}, 
\end{array} \right.
\end{align*}
for $\bx=(1,1)$,
\begin{align*}
U_{(1,1),\bx'} = \left\{ \begin{array}{ll} 
(A^\emptyset_{1,1})^\top, & \bx'=(0,0), \cr
(A^{\{1\}}_{1-x_1',1})^\top, & \bx'=(x_1',0)\in\{(1,0),(2,0)\}, \cr
(A^{\{2\}}_{1,1-x_2'})^\top, & \bx'=(0,x_2')\in\{(0,1),(0,2)\}, \cr
(A^{\{1,2\}}_{1-x_1',1-x_2'})^\top, & \bx'=(x_1',x_2')\in\{(1,1),(2,1),(2,2),(1,2)\}, 
\end{array} \right.
\end{align*}
for $\bx=(x_1,0)\in\mathbb{Z}_{\ge 2}\times\{0\}$,
\begin{align*}
U_{(x_1,0),\bx'} = \left\{ \begin{array}{ll} 
(A^{\{1\}}_{x_1-x_1',0})^\top, & \bx'-\bx=(x_1'-x_1,0)\in\{-1,0,1\}\times\{0\}, \cr
(A^{\{1,2\}}_{x_1-x_1',-1})^\top, &  \bx'-\bx=(x_1'-x_1,1)\in\{-1,0,1\}\times\{1\},
\end{array} \right.
\end{align*}
for $\bx=(x_1,1)\in\mathbb{Z}_{\ge 2}\times\{1\}$,
\begin{align*}
U_{(x_1,1),\bx'} = \left\{ \begin{array}{ll} 
(A^{\{1\}}_{x_1-x_1',1})^\top, & \bx'-\bx=(x_1'-x_1,-1)\in\{-1,0,1\}\times\{-1\}, \cr
(A^{\{1,2\}}_{x_1-x_1',1-x_2'})^\top, &  \bx'-\bx=(x_1'-x_1,x_2'-1)\in\{-1,0,1\}\times\{0,1\},
\end{array} \right.
\end{align*}
for $\bx=(0,x_2)\in\{0\}\times\mathbb{Z}_{\ge 2}$,
\begin{align*}
U_{(0,x_2),\bx'} = \left\{ \begin{array}{ll} 
(A^{\{2\}}_{0,x_2-x_2'})^\top, & \bx'-\bx=(0,x_2'-x_2)\in\{0\}\times\{-1,0,1\} \cr
(A^{\{1,2\}}_{-1,x_2-x_2'})^\top, &  \bx'-\bx=(1,x_2'-x_2)\in\{1\}\times\{-1,0,1\},
\end{array} \right.
\end{align*}
for $\bx=(1,x_2)\in\{1\}\times\mathbb{Z}_{\ge 2}$,
\begin{align*}
U_{(1,x_2),\bx'} = \left\{ \begin{array}{ll} 
(A^{\{2\}}_{1,x_2-x_2'})^\top, & \bx'-\bx=(-1,x_2'-x_2)\in\{-1\}\times\{-1,0,1\}, \cr
(A^{\{1,2\}}_{1-x_1',x_2-x_2'})^\top, &  \bx'-\bx=(x_1'-1,x_2'-x_2)\in\{0,1\}\times\{-1,0,1\},
\end{array} \right.
\end{align*}
and for $\bx=(x_1,x_2)\in\mathbb{Z}_{\ge 2}^2$ and for $ \bx'-\bx=(x_1'-x_1,x_2'-x_2)\in\{-1,0,1\}^2$,
\begin{align}
U_{\bx,\bx'} = (A^{\{1,2\}}_{x_1-x_1',x_2-x_2'})^\top.
\label{eq:Uxxp_def}
\end{align}
The block structure of $T^\top$ is slightly different from that of $T$, but the difference between them is restricted to near the boundaries and our results for the occupation measure can be applied to $T^\top$.
%
According to \eqref{eq:Uxxp_def}, define blocks $B^{\{1,2\}}_{i,j}$ for $i,j\in\{-1,0,1\}$ as $B^{\{1,2\}}_{i,j}=(A^{\{1,2\}}_{-i,-j})^\top$ and let $B^{\{1,2\}}_{*,*}(z_1,z_2)$ be the matrix-valued generating function corresponding to $A^{\{1,2\}}_{*,*}(z_1,z_2)$ of $T$, defined as 
\[
B^{\{1,2\}}_{*,*}(z_1,z_2) 
= \sum_{i_1,i_2\in\{-1,0,1\}} z_1^{i_1} z_2^{i_2} B^{\{1,2\}}_{i_1,i_2}
= (A^{\{1,2\}}_{*,*}(z_1^{-1},z_2^{-1}))^\top. 
\]
For $x_1,x_2\in\mathbb{R}_+\setminus\{0\}$, we have 
\begin{equation}
\cp(B^{\{1,2\}}_{*,*}(x_1,x_2)) = \cp((A^{\{1,2\}}_{*,*}(x_1^{-1},x_2^{-1}))^\top) = \cp(A^{\{1,2\}}_{*,*}(x_1^{-1},x_2^{-1})), 
\end{equation}
and this leads us to 
\begin{align}
\Gamma_R^{\{1,2\}} 
&= \{(\theta_1,\theta_2)\in\mathbb{R}^2; \cp(A^{\{1,2\}}_{*,*}(e^{-\theta_1},e^{-\theta_2}))>1 \} \cr
&=\{(\theta_1,\theta_2)\in\mathbb{R}^2; \cp(B^{\{1,2\}}_{*,*}(e^{\theta_1},e^{\theta_2}))>1 \}. 
\end{align}
Hence, we see that the region $\Gamma_R^{\{1,2\}}$ for $T^\top$ corresponds to the region $\Gamma^{\{1,2\}} $ for $T$.
%
For $i\in\{-1,0,1\}$, define $\bar B^{\{1\}}_i$ and $\bar B^{\{2\}}_i$ in the same way as that used in defining $\bar A^{\{1\}}_i$ and $\bar A^{\{2\}}_i$ in Section \ref{sec:intro}. To be precise, 
\begin{align*}
&\bar B^{\{1\}}_i = \left( \bar B^{\{1\}}_{i,(x_2,x_2')}; x_2,x_2'\in\mathbb{Z}_+ \right),\\
&\bar B^{\{1\}}_{i,(x_2,x_2')} = \left\{ \begin{array}{ll}
A^{\{1\}}_{-i,0}, & \mbox{if $x_2=0$ and $x_2'=0$}, \cr
A^{\{1,2\}}_{-i,-1}, & \mbox{if $x_2=0$ and $x_2'=1$}, \cr
A^{\{1\}}_{-i,1}, & \mbox{if $x_2=1$ and $x_2'=0$}, \cr
A^{\{1,2\}}_{-i,1-x_2'}, & \mbox{if $x_2=1$ and $x_2'-1\in\{0,1\}$}, \cr
A^{\{1,2\}}_{-i,x_2-x_2'}, & \mbox{if $x_2\ge 2$ and $x_2'-x_2\in\{-1,0,1\}$}, \cr
O, & \mbox{otherwise}, 
\end{array} \right. 
\end{align*}
\begin{align*}
&\bar B^{\{2\}}_i = \left( \bar B^{\{2\}}_{i,(x_1,x_1')}; x_1,x_1'\in\mathbb{Z}_+ \right),\\
&\bar B^{\{2\}}_{i,(x_1,x_1')} = \left\{ \begin{array}{ll}
A^{\{2\}}_{0,-i}, & \mbox{if $x_1=0$ and $x_1'=0$}, \cr
A^{\{1,2\}}_{-1,-i}, & \mbox{if $x_1=0$ and $x_1'=1$}, \cr
A^{\{2\}}_{1,-i}, & \mbox{if $x_1=1$ and $x_1'=0$}, \cr
A^{\{1,2\}}_{1-x_1',-i}, & \mbox{if $x_1=1$ and $x_1'-1\in\{0,1\}$}, \cr
A^{\{1,2\}}_{x_1-x_1',-i}, & \mbox{if $x_1\ge 2$ and $x_1'-x_1\in\{-1,0,1\}$}, \cr
O, & \mbox{otherwise}.
\end{array} \right.
\end{align*}
By the definition, we have $\bar B^{\{1\}}_i=(\bar A^{\{1\}}_{-i})^\top$ and $\bar B^{\{2\}}_i=(\bar A^{\{2\}}_{-i})^\top$. Let $\bar B^{\{1\}}_*(z_1)$ and $\bar B^{\{2\}}_*(z_2)$ be the matrix-valued generating functions corresponding to $\bar A^{\{1\}}_*(z_1)$ and $\bar A^{\{2\}}_*(z_2)$ of $T$, respectively, defined as
\[
\bar B^{\{1\}}_*(z_1) = \sum_{i\in\{-1,0,1\}} z_1^{i} \bar B^{\{1\}}_{i}=(\bar A^{\{1\}}_*(z_1^{-1}))^\top,\quad
\bar B^{\{2\}}_*(z_2) = \sum_{i\in\{-1,0,1\}} z_2^{i} \bar B^{\{2\}}_{i}=(\bar A^{\{2\}}_*(z_2^{-1}))^\top.
\]
We have 
\begin{align}
\Gamma_R^{\{1\}} 
&= \{(\theta_1,\theta_2)\in\mathbb{R}^2; \cp(\bar A^{\{1\}}_*(e^{-\theta_1}))>1\ \mbox{for some $\theta_2\in\mathbb{R}$} \} \cr
&=\{(\theta_1,\theta_2)\in\mathbb{R}^2; \cp(\bar B^{\{1\}}_*(e^{\theta_1}))>1\ \mbox{for some $\theta_2\in\mathbb{R}$} \}, \\
 \Gamma_R^{\{2\}} 
&= \{(\theta_1,\theta_2)\in\mathbb{R}^2; \cp(\bar A^{\{2\}}_*(e^{-\theta_2}))>1\ \mbox{for some $\theta_1\in\mathbb{R}$} \} \cr
&=\{(\theta_1,\theta_2)\in\mathbb{R}^2; \cp(\bar B^{\{2\}}_*(e^{\theta_2}))>1\ \mbox{for some $\theta_1\in\mathbb{R}$} \}.
\end{align}
Hence, we see that the regions $\Gamma_R^{\{1\}}$ and $\Gamma_R^{\{2\}}$ for $T^\top$ corresponds to the regions $\Gamma^{\{1\}}$ and $\Gamma^{\{2\}}$ for $T$, respectively.
As a result, replacing $\Gamma^{\{1\}} $, $\Gamma^{\{2\}}$ and $\Gamma^{\{1,2\}}$ with $\Gamma_R^{\{1\}}$, $\Gamma_R^{\{2\}}$ and $\Gamma_R^{\{1,2\}}$, respectively, we obtain the asymptotic properties of the hitting measure by using the results in Section \ref{sec:asymptotics_o}. 

%
Here we mention only the asymptotic decay rates in directions $\bc=(1,0)$ and $\bc=(0,1)$ for the hitting measure. 
Consider the following optimization problem.
\begin{align}
\begin{array}{ll}
\mbox{Maximize} & s_1+s_2, \cr
\mbox{subject to} & s_1 \le \sup\left\{\theta_1\in\mathbb{R}; (\theta_1,\theta_2)\in \Gamma_R^{\{1\}}\cap\Lambda^{\{2\}}(s_2)\cap\Gamma_R^{\{1,2\}}\ \mbox{for some $\theta_2\in\mathbb{R}$} \right\}, \cr
&  s_2 \le \sup\left\{\theta_2\in\mathbb{R}; (\theta_1,\theta_2)\in\Lambda^{\{1\}}(s_1)\cap\Gamma_R^{\{2\}}\cap\Gamma_R^{\{1,2\}}\ \mbox{for some $\theta_1\in\mathbb{R}$} \right\}.
\end{array}
\end{align}
Under Assumption \ref{as:Gamma1212_notempty}, this optimization problem is also feasible. Since the regions $\Gamma_R^{\{1\}}$, $\Gamma_R^{\{2\}}$ and $\Gamma_R^{\{1,2\}}$ are convex sets, it has a unique optimal solution $(s_{R,1}^*,s_{R,2}^*)$, which satisfy 
\begin{align}
s_{R,1}^* 
&= -\inf\!\left\{\theta_1\in\mathbb{R}; (\theta_1,\theta_2)\in\Gamma^{\{1\}}\cap[\Gamma^{\{2\}}]^{ex}\cap\Gamma^{\{1,2\}} \right\}, 
\label{eq:sR1s_eq1}\\
s_{R,2}^* 
&= -\inf\!\left\{\theta_2\in\mathbb{R}; (\theta_1,\theta_2)\in[\Gamma^{\{1\}}]^{ex}\cap\Gamma^{\{2\}}\cap\Gamma^{\{1,2\}} \right\}.  
\label{eq:sR2s_eq1}
\end{align}
The asymptotic decay rate of the hitting measure in direction $\bc=(1,0)$ is given by $s_{R,1}^*$ and that in direction $\bc=(0,1)$ by $s_{R,2}^*$.

%
%
\section{Concluding remark} \label{sec:conclusion}

We defined the occupation measure and hitting measure for a nonnegative matrix having the same block structure as the transition probability matrix of a 2d-QBD process, and obtained their asymptotic properties. In the process of analysis, we used the G-matrix functions $G_1(z)$ and $G_2(z)$, defined in Section \ref{sec:mgf_potential}.
If $z$ is a positive real number, each of $(G_1(z)^n,n\ge 1)$ and $(G_2(z)^n,n\ge 1)$ corresponds to the hitting measure for a nonnegative matrix having the same block structure as the transition probability matrix of a one-dimensional QBD process. In other words, we used one-dimensional QBD-type models to analyze a two-dimensional QBD-type model. 
Our approach used in the paper is useful in studying asymptotics of the stationary distribution in a three-dimensional reflecting random walk (3d-RRW for short) \cite{Ozawa25b}, where the asymptotic properties of the hitting measure for a nonnegative matrix of 2d-RRW type are used. A 2d-QBD process is a 2d-RRW with a background process and our results can be applied to nonnegative matrices of 2d-RRW type.

%
%

\end{document}